\documentclass[a4paper, 11pt]{paper}
\usepackage{amssymb}
\usepackage{amsmath}
\usepackage{amsthm}
\usepackage{enumerate}
\usepackage[mathscr]{eucal}
\usepackage{graphics}
\usepackage[dvips]{graphicx}
\usepackage{inctabs}

\begin{document}



\let\goth\mathfrak


\newcommand*{\Land}{\;\land\;}
\newcommand\theoremname{Theorem}
\newcommand\lemmaname{Lemma}
\newcommand\corollaryname{Corollary}
\newcommand\propositionname{Proposition}
\newcommand\factname{Fact}
\newcommand\remarkname{Remark}
\newcommand\examplename{Example}

\newtheorem{thm}{\theoremname}[section]
\newtheorem{lem}[thm]{\lemmaname}
\newtheorem{cor}[thm]{\corollaryname}
\newtheorem{prop}[thm]{\propositionname}
\newtheorem{fact}[thm]{\factname}
\newtheorem{exmx}[thm]{\examplename}
\newenvironment{exm}{\begin{exmx}\normalfont}{\end{exmx}}

\newtheorem{rem}[thm]{\remarkname}

\def\myend{{}\hfill{\small$\bigcirc$}}

\newtheorem{reprx}[thm]{Representation}
\newenvironment{repr}{\begin{reprx}\normalfont}{\myend\end{reprx}}
\newtheorem{cnstrx}[thm]{Construction}
\newenvironment{constr}{\begin{cnstrx}\normalfont}{\myend\end{cnstrx}}
\def\classifname{Classification}
\newtheorem{classification}[thm]{\classifname}
\newenvironment{classif}{\begin{classification}\normalfont}{\myend\end{classification}}

\bibliographystyle{acm}
\newcommand{\minus}{{\bf -}}

\title{Multiplied configurations, series induced by quasi difference sets}
\author{Krzysztof Petelczyc \and Krzysztof Pra\.zmowski}
\pagestyle{myheadings}
\markboth{K. Petelczyc, K. Pra{\.z}mowski}{Multiplied configurations}

\def\LineOn(#1,#2){\overline{{#1},{#2}\rule{0em}{1,5ex}}}
\def\PointOf(#1,#2){{#1}\sqcap{#2}}
\def\lines{{\cal L}}
\def\collin{\sim}
\def\chain(#1){{#1}^{\ast}}
\def\inc{\mathrel{\rule{2pt}{0pt}\rule{1pt}{9pt}\rule{2pt}{0pt}}}
\def\dist{\mathrm{dist}}
\def\DifSpace(#1,#2){{\bf D}({#1},{#2})}
\def\CyclSpacex(#1,#2,#3){{#1}\circledast_{{#2}}{#3}}
\def\CyclSpace(#1,#2,#3){\circledast_{{#1}}({#2},{#3})}
\def\entier{\mathrm{E}}
\def\invers(#1){{#1}^{{-1}}}
\def\nwd(#1,#2){\mathrm{GCD}(#1,#2)}
\def\gras(#1,#2){{\bf G}_{#1}({#2})}
\def\otocz(#1,#2){{#1}_{({#2})}}
\def\alf(#1,#2){\mbox{${\strut}^{\alpha}\mkern-2mu{#1}_{({#2})}$}}
\def\bet(#1,#2){\mbox{${\strut}^{\beta}\mkern-2mu{#1}_{({#2})}$}}
\def\embfunc{\varepsilon}
\def\emb(#1,#2){\embfunc_{{#1}}({#2})}
\def\img{\mathrm{im}}
\def\ciach#1{\;\wr_{{#1}}\;}
\def\dod{\mathrel{\wr}}
\def\baero{\mathrel{\rho}}
\def\kor{\mathrel{\varkappa}}
\def\dod{\mathrel{\wr}}
\def\baer{\mathrel{\diamond}}
\def\Baer{\mathrel{\scriptstyle{\blacklozenge}}}
\def\corr{\mathrel{\vartriangleleft}}
\def\pls{partial linear space}
\def\id{\mathrm{id}}
\def\mappedby#1{%
{\rule{0pt}{2.5ex}}\mkern8mu{\longmapsto\mkern-25mu{\raise1.3ex\hbox{$#1$}}}
\mkern20mu{\rule{0pt}{4pt}}
}
\def\rank(#1){\mathrm{r}(#1)}
\def\qdots{\rule{0pt}{10pt}%
 {\raise-0.6ex\hbox{$\cdot$}}\mkern1.2mu{\raise0.1ex\hbox{$\cdot$}}%
 \mkern1.2mu{\raise0.8ex\hbox{$\cdot$}}
 }
\def\supp{\mathrm{supp}}
\def\flines{{\cal J}}  
\newcommand*{\struct}[1]{{\ensuremath{\langle #1 \rangle}}}
\newcommand*{\sub}{\raise.5ex\hbox{\ensuremath{\wp}}}
\def\Aut{{\text{\rm Aut}}}

\def\aulemname{Auxiliary Lemma}
\newtheorem{auxlem}{\aulemname}[thm]
\newenvironment{auproof}{\par\noindent{\it D o w ó d}:}{\hfill$\vartriangle$\par}
\newcounter{sentencex}
\def\thesentencex{\Roman{sentencex}}
\def\labelsentencex{\upshape(\thesentencex)}

\newenvironment{sentencesx}{%
        \list{\labelsentencex}
          {\usecounter{sentencex}\def\makelabel##1{\hss\llap{##1}}
            \topsep3pt\leftmargin0pt\itemindent40pt\labelsep8pt}%
  }{%
    \endlist}

\newenvironment{cmath}{%
  \par
  \smallskip
  \centering
  $
}{%
  $
  \par
  \smallskip
  \csname @endpetrue\endcsname
}


  \def\pforall#1{(\forall{#1})}
  \def\pexists#1{(\exists{#1})}

\newcounter{sentence}
\def\thesentence{\roman{sentence}}
\def\labelsentence{\upshape(\thesentence)}

\newenvironment{sentences}{%
        \list{\labelsentence}
          {\usecounter{sentence}\def\makelabel##1{\hss\llap{##1}}
            \topsep3pt\leftmargin0pt\itemindent40pt\labelsep8pt}%
  }{%
    \endlist}

\newenvironment{ctext}{%
  \par
  \smallskip
  \centering
}{%
 \par
 \smallskip
 \csname @endpetrue\endcsname
}
\def\LineOn(#1,#2){\overline{{#1},{#2}\rule{0em}{1,5ex}}}
\def\PointOf(#1,#2){{#1}\sqcap{#2}}
\def\lines{{\cal L}}
\def\collin{\sim}
\def\chain(#1){{#1}^{\ast}}
\def\inc{\rule{2pt}{0pt}\rule{1pt}{9pt}\rule{2pt}{0pt}}
\def\dist{\mathrm{dist}}
\def\DifSpace(#1,#2){{\bf D}({#1},{#2})}
\def\CyclSpace(#1,#2){\circledast_{{#1}} {#2}}
\def\CyclSpacex(#1,#2,#3){{#1}\circledast_{{#2}}{#3}}
\def\entier{\mathrm{E}}
\def\invers(#1){{#1}^{{-1}}}
\def\nwd(#1,#2){\mathrm{GCD}(#1,#2)}
\def\gras(#1,#2){{\bf G}_{#1}({#2})}
\def\otocz(#1,#2){{#1}_{({#2})}}
\def\alf(#1,#2){\mbox{${\strut}^{\alpha}\mkern-2mu{#1}_{({#2})}$}}
\def\bet(#1,#2){\mbox{${\strut}^{\beta}\mkern-2mu{#1}_{({#2})}$}}
\def\embfunc{\varepsilon}
\def\emb(#1,#2){\embfunc_{{#1}}({#2})}
\def\img{\mathrm{im}}
\def\ciach#1{\;\wr_{{#1}}\;}

\def\mappedby#1{%
{\rule{0pt}{2.5ex}}\mkern8mu{\longmapsto\mkern-25mu{\raise1.3ex\hbox{$#1$}}}
\mkern20mu{\rule{0pt}{4pt}}
}

\def\qdots{\rule{0pt}{10pt}%
 {\raise-0.6ex\hbox{$\cdot$}}\mkern1.2mu{\raise0.1ex\hbox{$\cdot$}}%
 \mkern1.2mu{\raise0.8ex\hbox{$\cdot$}}
 }

\def\QDS{{\sf QDS}}

\maketitle

\begin{abstract}
  Using the technique of quasi difference sets we characterize geometry
  and automorphisms of configurations which can be presented as a join of
  some others, in particular -- which can be presented as series of
  cyclically inscribed copies of another configuration.\\
MSC 2000: 51D20, 51E30 \\
Key words: partial linear space, difference set, quasi difference set, cyclic projective plane.
\end{abstract}

\section*{Introduction}

The technique involving difference sets is one of the standard ones used
to construct block designs of various types, in particular -- to construct
finite projective planes.
In fact, every finite Desarguesian projective plane  can be defined with
the help of this method (see \cite{Lipski}).
However, in effect, the structure defined in terms of a difference set
cannot be "partial".
We propose to overcome this trouble generalizing the notion of difference set 
to a {\em quasi difference set}.
\newline
This approach, applied in \cite{petel} to a very special and simple case of
products of two cyclic groups could be fruitfully used to represent configurations,
which can be visualized as series of closed polygons,
inscribed cyclically one into the previous one.
In particular, classical Pappus configuration can be presented in this way,
and some others as well.

\par
The idea is simple -- blocks ("lines") are the images of some fixed subset $D$
of a group $\sf G$ under left translations of this group.
Some necessary and sufficient conditions are imposed on $D$
which assure that the resulting incidence structure is a $\lambda$-design
(such a set $D$ is called a difference set in $\sf G$);
specifically, for $\lambda = 1$ -- a linear space.
Some weaker conditions imposed on $D$ assure that the resulting structure
is a partial linear space. A set $D$ which meets these conditions is called
a quasi difference set in $\sf G$.

\par
Some other generalizations can be found in the literature.
One of them is the notion of a {\em partial difference set} (PDS).
While defining a difference set $D$ we require that every nonzero element of 
$\goth G$ can be presented in exactly $\lambda$ ways as a difference 
of two elements of $D$, defining a PDS $D$ we require that
every nonzero element $a$ of $\goth G$ can be presented in 
$\lambda_1$ ways (if $a \in D$) and in $\lambda_2$ ways (if $a \notin D$)
as a corresponding difference, for some fixed $\lambda_1,\lambda_2$
(cf. e.g. \cite{PDS1}).
Generally, a nonzero element of $\goth G$ can be presented in 
either $\lambda_1 = 0$, or $\lambda_2 = 1$, or $\lambda_3 = n$ ways
as a difference
of a $n$-element quasi difference set.
However, no algebraic regularity is assumed characterizing those elements of 
$\goth G$, which are $\lambda_i$-ways differences, for each particular $i =1,2,3$.
Moreover, admitting elements which are $\lambda_3 = n$-ways differences
we come to (geometrically) less regular structures so,
finally, in this paper we consider quasi difference sets with only 
$\lambda_1 = 0$ and $\lambda_2 = 1$ admitted.
However, such quasi difference sets {\em are not} partial difference sets.
One of the most important questions in the theory of PDS's is 
to determine the existence and characterize such sets in various particular
groups, for various special types of them (there is a huge literature 
on this subject, older and newer,
see e.g. \cite{eliot}, \cite{chandler}, \cite{gordon}).
These are not the questions considered in 
this paper.
Instead, we are mainly interested in the geometry (in the rather classical style) of 
partial linear spaces determined by quasi difference sets. 

\par
This project was started in \cite{petel}.
Here, we study in some details configurations which can be defined with the help
of arbitrary quasi difference set.
We also pay some attention to elementary properties of such structures:
we discuss if Veblen, Pappos, and Desargues axioms may hold
in them (Prop.'s \ref{pr:veblen}, \ref{pr:desargues}, \ref{lem:pappus},
\ref{cor:vebdes}, \ref{prop:veb3}, \ref{prop:des3}).
A special emphasis was imposed on structures which arise from groups
decomposed into a cyclic group $C_k$ and some other group $\sf G$, simply because
these structures can be seen as {\em multiplied configurations} --
series of cyclically inscribed configurations, each one isomorphic to the
configuration associated with $\sf G$.
This construction, on the other hand, is just a special case of the operation of 
"joining" ("gluing")
two structures, corresponding to the operation of the direct sum of groups.
In some cases corresponding decomposition can be defined within the resulting
"sum", in terms of the geometry of the considered structures.
This definable decomposition enables us to characterize the automorphism group
of such a "glue-sum".
Some other techniques are used to determine the automorphism group of 
cyclically inscribed configuration.
Roughly speaking, groups in question are semidirect products of 
some symmetric group and the group of translations of the underlying group.

\par
The technique of quasi difference sets can be used to produce new configurations,
so far not considered in the literature.
Many of them seem to be of a real geometrical interest for their own.
In the last section we apply our apparatus to some new configurations,
arising from the well known (like cyclically inscribed Pappus or Fano configurations,
 multiplied Pappus configurations, sums of cyclic projective planes),
and determine their geometric properties and automorphisms.
\par
Usually, dealing with abelian groups we shall follow "additive" notation,
while the "multiplicative" one will be used for arbitrary group.

\section{Basic notions}

In \cite{petel} series of cyclically inscribed $n$-gons were investigated
and for this purpose a construction involving quasi difference sets was used.
Below we briefly recall this construction.
Let ${\sf G} = \struct{G,\cdot,1}$ be an arbitrary group and $D\subset G$, we set
$$\lines = \lines_{({\sf G},D)} = G/D = \{a\cdot D\colon a\in G\}.$$
Clearly, $\lines \subseteq \sub(G)$, and, since every left translation 
$\tau_a\colon G\ni x\mapsto a\cdot x\in G$ is a bijection, we get
$|a\cdot D| = |D|$ for every $a\in G$.
Following this notation we can write $a\cdot D = \tau_a(D)$, and
$\lines_{({\sf G},D)} = \{ \tau_a(D)\colon a\in G \}$.
We set
\begin{equation}\label{def:mystructure}
  \DifSpace({\sf G},D) = \struct{G,\lines_{({\sf G},D)}}.
\end{equation}
Every translation $\tau_a$ over $\sf G$ is an automorphism of $\DifSpace({\sf G},D)$.
Indeed, clearly, 
$\lines_{({\sf G},D)} = \lines_{({\sf G},\tau_a(D))}$ 
for every $a \in G$;
this also yields that without loss of generality we can assume that $1 \in D$.

It was proved in \cite{petel} that the following conditions are equivalent:
\begin{quotation}
\begin{description}
\item
  The structure $\DifSpace({\sf G},D)$ is a configuration (i.e. a partial linear space
  in which the rank of a point and the rank of a line are equal)
\item[\QDS:]
  For every $c\in G$, $c\neq 1$ there is at most one pair $(a,b)\in D\times D$ with $a\invers(b)=c$.
\end{description}
\end{quotation}
If $\sf G$ is abelian then $\DifSpace({\sf G},D)$ is a configuration iff it is
a partial linear space with point rank at least 2.
The number of points of $\DifSpace({\sf G},D)$ is $|G|$, the number of lines
is $\frac{|G|}{|G_D|}$, the rank of a line is $|D|$, and the rank of a point is
$\frac{|D|}{|G_D|}$.

In the sequel {\em a quasi difference set in} $\sf G$ means any subset $D$ of $G$
which satisfies \QDS.
In \cite{petel} we were mainly interested in the structures of
the form 
$\DifSpace(C_k\oplus C_n,{\cal D})$, where 
${\cal D} = \{ (0,0),(1,0),(0,1) \}$.
In the paper we shall generalize this construction.

Let us adopt the following convention:
\begin{itemize}
\item
  element $a\in G$ will be denoted as $p_a$ -- an abstract "point" with
  "coordinates" $(a)$;
\item
  a line $a\cdot D\in G/D$ will be denoted by $l_a$ -- its "coordinates" will
  be written as $[a]$.
\end{itemize}
Then we can write
\begin{equation}\label{wz:analinc}
  (a)\inc [b] \quad\text{iff}\quad p_a \in l_b \quad\text{iff}\quad a \in b \cdot D 
  \quad\text{iff}\quad \invers(b)\cdot a \in D.
\end{equation}
We use the symbol $\inc$ to denote the relation of incidence.


\par
Generally, an automorphism of an incidence structure 
  $\goth M= \struct{S,\lines,\inc}$ 
is a pair 
  $\varphi=(\varphi',\varphi'')$ of bijections 
$\varphi': S \longmapsto S$, $\varphi'':\lines \longmapsto \lines$ such that
for every $a \in S$, $l \in \lines$ the conditions 
  $a \inc l$ and $\varphi'(a) \inc \varphi''(l)$ are equivalent.
In particular, if $\goth M =\DifSpace({\sf G},D)$ then every automorphism 
$\varphi=(\varphi',\varphi'')$ of $\goth M$ uniquely corresponds to a pair 
$f=(f',f'')$ of bijections of $G$ determined by 
  $$\varphi'\big((a)\big)= \big(f'(a)\big), \; 
     \varphi''\big([b]\big)= \big[f''(b)\big].$$
We shall frequently refer to the pair $f$ as to an automorphism of $\goth M$.

\par
As a convenient tool to establish possible automorphisms of an incidence structure 
${\goth M} = \struct{S,\lines,\inc}$,
we frequently use in the paper the notion of the neighborhood 
$\otocz({\goth M},a)$ of a point $a \in S$.
It is a substructure of $\goth M$, whose points are all the points of $\goth M$
collinear with $a$, and lines are the lines of $\goth M$ which contain at least
two points collinear with $a$
("lines" are considered in a purely incidence way here: 
lines of $\otocz({\goth M},a)$ consist only of points of $\otocz({\goth M},a)$).
Clearly, if $\varphi=(\varphi',\varphi'')$ is an automorphism of $\goth M$,
then $\varphi$ maps $\otocz({\goth M},a)$ onto $\otocz({\goth M},{\varphi'(a)})$.

\par
Let ${\goth M}=\struct{M,\lines}$  be an arbitrary partial linear space
and let $\varkappa = (\varkappa',\varkappa'')$ 
with 
  $\varkappa'\colon M \longrightarrow\lines$, 
  $\varkappa''\colon\lines \longrightarrow M$
be a correlation of ${\goth M}'$.
Recall that if $\varkappa = (\varkappa',\varkappa'')$ is a correlation of 
$\goth M$ then the pair $\psi$ of maps,
(convention: coordinate-wise composition of pairs of functions)
$\psi=(\psi',\psi'')=(\varkappa'',\varkappa')\circ(\varkappa',\varkappa'')$
is a {\em standard} collineation of $\goth M$.



\section{Generalities}

Now, we are going to present some general facts about properties of the structure
$\DifSpace({\sf G},D)$.

Some of the automorphisms of $\DifSpace({\sf G},D)$ are determined
by automorphisms of the underlying group $\sf G$, namely
\begin{rem}\label{rem:whenautline}
  If $f\in \Aut({\sf G})$ then $f$ determines an automorphism 
  $\varphi=(\varphi',\varphi'')$ of $\DifSpace({\sf G},D)$ with $\varphi'=f$
  iff $f(D)=q\cdot D$ for some $q\in G$  and 
  $\varphi''([a])=[q\cdot f(a)]$ for every $a \in G$ . 
\end{rem}

\begin{prop}\label{pr:corelindif}
  Let $D$ be a quasi difference set in a commutative group $\sf G$. Then 
  the map $\varkappa$ defined by
  \begin{equation}\label{wz:corel}
    \varkappa((a))=[\invers(a)], \;\;
    \varkappa([a])=(\invers(a))
  \end{equation}
  is an involutive correlation of
  the structure ${\goth D} = \DifSpace({\sf G},D)$.
  Consequently, $\goth D$ is self-dual.
    A point $a$ of $\goth D$ is selfconjugate under $\varkappa$
    iff $a^2 \in D$.
\end{prop}
\begin{proof}
  Clearly, $\varkappa$ is involutory.
 \par
  Let $a,b\in G$. 
  Then $(a)\inc[b]$ means that $\invers(b)\cdot a \in D$.
  This is equivalent to $\invers({(\invers(a))})\cdot \invers(b)\in D$, i.e.
  $\varkappa([b]) = (\invers(b))\inc[\invers(a)]= \varkappa((a))$.
  Thus $\varkappa$ is a correlation.

  Finally, assume that $(a)\inc\varkappa((a))=[\invers(a)]$.
  From \eqref{wz:analinc} we obtain $a^2\in D$.
\end{proof}
The correlation defined by \eqref{wz:corel} will be referred to
as {\em the standard correlation of $\DifSpace({\sf G},D)$}.

Following a notation frequently used in the projective geometry
we write $\chain([b])$ for the set of the points which are incident with a line $[b]$,
  $$\chain([b]) = \{ (b\cdot d)\colon d\in D \}.$$
Immediate from \eqref{wz:analinc} is the following
\begin{lem}\label{lem:analchain}
  Let ${\goth D}=\DifSpace({\sf G},D)$ for
  a subset $D$ of a group $\sf G$ and let  $a,b\in G$.
  \begin{sentences}\itemsep-2pt
  \item\label{analchain:cas1}
    The set of the lines of $\goth D$ through the point $(a)$ 
    can be identified with $a\cdot\invers(D)$.
    We write $\chain({(a)})$ for the set of the lines through $(a)$ and then
    \begin{equation}\label{wz:analchain1}
      \chain({(a)}) = \{ [a\cdot\invers(d)]\colon d\in D \}.
    \end{equation}
    %
  \item\label{analchain:cas2}
    The points $(a)$ and $(b)$ are collinear in $\goth D$ 
    (we write: $(a)\collin(b)$) iff 
       $\invers(a)\cdot b\in \invers(D)D$.
    If $(a)\neq (b)$ are two collinear points then we write 
    $\LineOn(a,b)$ for the line which joins these two points.
    If $\invers(a)b = \invers(d_1) d_2$ with $d_1,d_2\in D$ then
    \begin{equation}\label{wz:analchain2}
      \LineOn(a,b) = [a\cdot\invers(d_1)] = [b\cdot\invers(d_2)].
    \end{equation}
  \item\label{analchain:cas3}
    The lines $[a]$ and $[b]$ of $\goth D$ have a common point iff 
    $\invers(a)\cdot b \in D\invers(D)$.
    We write $\PointOf(a,b)$ for the common point of two mutually intersecting
    lines $[a]$ and $[b]$.
    If $\invers(a)\cdot b = d_1\cdot\invers(d_2)$ with $d_1,d_2\in D$ then
    \begin{equation}\label{wz:analchain3}
      \PointOf(a,b) = (a\cdot d_1) = (b\cdot d_2).
    \end{equation}
  \end{sentences}
\end{lem}
\begin{proof}
\eqref{analchain:cas1}:
  By definition, $(a)\inc[b]$ is equivalent to $a\in bD$, i.e. to
  $\invers(a)\in \invers({(bD)}) = \invers(D)\invers(b)$.
  And the last condition is equivalent to $\invers(a)b\in\invers(D)$,
  i.e. to $b\in a\invers(D)$.

\eqref{analchain:cas2}:
  From \eqref{analchain:cas1}, $(b)$ is collinear with $(a)$ iff $(b)\inc[p]$ for some line
  $[p]$ through $(a)$. This means that $b = pd_2$, for some $d_2\in D$, and 
  $[p]\in\chain({(a)})$.
  Thus $b = a\invers(d_1)d_2$ for some $d_1,d_2\in D$, 
  so $\invers(a)\cdot b \in \invers(D)D$. 
  Then $a\invers(d_1) = b\invers(d_2)$ and
  directly from \eqref{wz:analinc} we verify that 
  $(a),(b)\inc[a\invers(d_1)]$.

\eqref{analchain:cas3} is proved dually; $[a]$ and $[b]$ have a common point if 
  $ad_1 = bd_2$ for some $d_1,d_2\in D$, which yields our claim.
\end{proof}

As a straightforward consequence of \ref{lem:analchain} we get 
\begin{prop}\label{pr:concomp}
  Let ${\sf G}=\struct{G,\cdot,1}$ be a group,
  $D$ be a quasi difference set in $\sf G$ with $1\in D$,
  and $a_1,a_2\in G$.
  Points $(a_1)$ and $(a_2)$ can be joined with a 
  polygonal path in $\DifSpace({\sf G},D)$ iff there is a finite sequence
  $q_1,\ldots,q_s$ of elements of $\invers(D) D$ such that 
  $a_1=q_1\cdot\ldots\cdot q_s\cdot a_2$.
  Consequently, the connected component of the point $(1)$ is isomorphic
  to $\DifSpace(\struct{D}_{\sf G},D)$, where $\struct{D}_{\sf G}$
  is the subgroup of $\sf G$ generated by $D$. Every two connected components of 
  any two points are isomorphic.
\end{prop}

\noindent
From now on we assume that $D$ generates $\sf G$.

\begin{equation}\label{ass:veblen}
 \begin{tabular}{c}
   Let $D$ be a quasi difference set in a group $\sf G$ satisfying
   the following:\\
if $d_1,d_2,d_3,d_4\in D$ and $d_1\invers(d_2)d_3\invers(d_4)\in D\invers(D)$
    then \\ $d_1\invers(d_2)=1$ or $d_3\invers(d_4)=1$ or $d_1\invers(d_4)=1$,
    or $d_3\invers(d_2)=1$.
 \end{tabular}
 \tag{$\bigstar$}
\end{equation}
\begin{lem}\label{lem:veblen}
 Assume that $\sf G$ is an abelian group, and \eqref{ass:veblen} holds.
  Let  $(a)$ be a point of $\DifSpace({\sf G},D)$, $d_1,d_2\in D$, and
  $[b_1]=[a\invers(d_1)]$ and $[b_2]=[a\invers(d_2)]$ be two distinct lines
  through $(a)$.
  \begin{sentences}
  \item\label{veblen:cas1}
    If $d'_i \in D$ and $(p_i)=(a\invers(d_i)d'_i)$
    is a point (on $[b_i]$) distinct from $(a)$ for $i=1,2$,
    then
    $(p_1)\collin(p_2)$ iff $d'_1=d'_2$.
  \item\label{veblen:cas2}
    For every point $(p_d)=(a\invers(d_1)d)$ of $[b_1]$
    with $d \in D$, $d\neq d_2,d_1$ there is the unique point $(q)=(a\invers(d_2)d)$ on
    $[b_2]$ which completes $(a),(p_d)$ to a triangle.
    The point $(p_{d_2})$ cannot be completed in such a way.
  \item\label{veblen:cas3}
    If $[g]$ is a line of $\DifSpace({\sf G},D)$ which crosses
    $[b_1]$, $[b_2]$ and misses $(a)$ then 
    $g=a\invers(d_1)\invers(d_2)d$ for some $d\in D$ with $d\neq d_1,d_2$.
  \item\label{veblen:cas4}
    If $[g_i]=[a\invers(d_1)\invers(d_2)d'_i]$ with $d'_i\in D$ for $i=1,2$ are two lines,
    both crossing $[b_1],[b_2]$ and missing $(a)$, then 
    $[g_1],[g_2]$ intersect each other in the point
    $(a\invers(d_1)\invers(d_2)d'_1d'_2)$.
  \end{sentences}
\end{lem}
\begin{proof}
  Since $\sf G$ is commutative, we have $D^{-1} D = D D^{-1}$,
  and $(x y)^{-1} = x^{-1} y^{-1}$ for all elements $x,y$ of $\sf G$.
 \par 
  \eqref{veblen:cas1}:
  In view of \ref{lem:analchain}\eqref{analchain:cas2}, 
    $(p_1)\collin(p_2)$ iff $\invers(p_1)p_2 \in \invers(D)D$.
  Since 
    $\invers(p_1)p_2 =   
    \invers(d'_1)d_1\invers(d_2)d'_2$,
  we need        
    $\invers(d'_1)d_1\invers(d_2)d'_2 \in D\invers(D)$.
  From \eqref{ass:veblen} we get one of the following:
  \begin{itemize}\itemsep-2pt
  \item
    $d_1\invers(d_2)=1$: in this case $d_1=d_2$, and thus $[b_1]=[b_2]$,
    contrary to the assumptions.
  \item
    $d'_2\invers(d_2)=1$ or $\invers(d'_1)=d_1$: in this case $(p_2)=(a)$ or $(p_1)=(a)$.
  \item
    $\invers(d'_1)d'_2=1$: this is our claim.
  \end{itemize}
 \eqref{veblen:cas2} follows immediately from \eqref{veblen:cas1}.

  \eqref{veblen:cas3}: Let $[g]$ cross $[b_i]$ in a point $(p_i)$. From 
  definition, $p_i=a\invers(d_i)d$ for some $d\in D$,
  and from \eqref{veblen:cas2}, $d\neq d_1,d_2$.
  Then from \eqref{wz:analchain3} we obtain
  $\LineOn(p_1,p_2) = [p_1\invers(d_2)]=[a\invers(d_1)\invers(d_2)d]$.

  \eqref{veblen:cas4}: From \ref{lem:analchain}\eqref{analchain:cas3} we find
  that $[g_1]$ and $[g_2]$ have a common point, since 
  $\invers(g_1)g_2 = d'_2\invers(d'_1)\in D\invers(D)$.
  With \eqref{wz:analchain3} we have
  $\PointOf(g_1,g_2)=(g_1 d'_2)=(a\invers(d_1)\invers(d_2)d'_1d'_2)$
\end{proof}

Then we shall try to establish automorphisms of structures
of the form \eqref{def:mystructure}.
Let us begin with some rigidity properties.
\begin{lem}\label{lem:rigidotocz}
  Let ${\goth D}=\DifSpace({\sf G},D)$, where $D$ satisfies assumption 
  \eqref{ass:veblen}.
  Let $f$ be a collineation of $\goth D$ such that 
  $f(o)=o$ for a point $o$.
  If $f$ satisfies any of the following:
  \begin{enumerate}[a)]\itemsep-2pt
  \item
    $f$ fixes all points on a line through $o$, or
  \item
    $f$ preserves every line through $o$
  \end{enumerate}
  then $f$ fixes all the points on lines through $o$.
\end{lem}
\begin{proof}
  Since $\goth D$ has a transitive group of automorphisms, we can assume 
  that $o=(d_0)=(0)$. Let $f=(f',f'')$ be a collineation of $\goth D$ and
  $f'\restriction{l_1} = \id_{l_1}$ for a line $l_1=[-d_1]$ through
  $o$ ($d_1\in D$). Take any line $l_2 = [-d_2]$. 
  Then the points $(-d_1+d_2)$ on $l_1$ and $(-d_2+d_1)$ on $l_2$
  give the unique pair of not collinear points "between" $l_1$ and $l_2$
  (cf. \ref{lem:veblen}(ii)).
  We have 
    $f'(-d_1+d_2)=(-d_1+d_2)$ 
  and
    $f'(-d_2+d_1)\inc[f''(-d_2)]$.
  The only point in $\otocz({\goth D},o)$ non-collinear with $(-d_1+d_2)$
  lies on $[-d_2]$;
  therefore $f''(-d_2)=-d_2$.
  Thus $f$ preserves every line through $o$.

  Now, let $f(l)=l$ for every line $l$ through $o$.
  Take arbitrary $l_1=[-d_1]$, $l_2=[-d_2]$, $d_1,d_2\in D$.
  As above, we show that $(-d_1+d_2)$ and $(-d_2+d_1)$ are preserved by $f'$.
  This yields $f'(-d_1+d_2)=(-d_1+d_2)$ for all $d_1,d_2\in D$, which is our 
  claim.
\end{proof}
As an immediate  consequence we get
\begin{cor}\label{cor:rigidotocz}
  Let $f$ be a collineation which fixes a line $l$ of $\goth D$ point-wise. 
  Under assumption  \eqref{ass:veblen} 
  $f$ fixes all points on every line which crosses $l$.
  Consequently, if $\goth D$ is connected then $f = \id$.
\end{cor}
\begin{cor}\label{cor:fixautisomor}
  Under assumption  \eqref{ass:veblen} 
  every automorphism of $\goth D$ which has a fixed point $o$ is uniquely 
  determined by its action on the lines through $o$.
  Consequently, the point stabilizer $\otocz({\Aut({\goth D})},o)$ of the automorphism group
  of $\goth D$ is isomorphic to a subgroup of $S_{r}$, where $r=|D|$.
\end{cor}


\section{Products of difference sets}

Let ${\sf G}_i = \struct{G_i,\cdot_i,1_i}$ be a group for $i\in I$,
and $1_i\in D_i\subset G_i$ for every $i\in I$.

Let $G=\prod_{i\in I} G_i$, i.e. let $G$ be the set of all functions
$g\colon I\longrightarrow \bigcup\{ G_i\colon i\in I \}$ with $g(i)\in G_i$.
Then the product $\prod_{i\in I}{\sf G}_i$ is the structure
$\struct{G,\cdot,1}$, where $(g_1\cdot g_2)(i)=g_1(i)\cdot_i g_2(i)$
for $g_1,g_2\in G$, and $1(i)=1_i$.
It is just the standard construction of the direct product of groups.
The set
\\[1.5ex]
\centerline{
  $\sum_{i\in I}G_i = \{g\in G\colon g(i)\neq 1_i \text{ for a finite number of }i\in I \}$}
\vskip1.5ex\noindent
is a subgroup of $\prod_{i\in I}{\sf G}_i$, denoted by 
  $\sum_{i\in I}{\sf G}_i$.
If $I=\{ 1,\ldots,r \}$ is finite then 
  ${\sf G}_1\oplus\ldots\oplus{\sf G}_{r}:= \prod_{i\in I}{\sf G}_i = 
    \sum_{i\in I}{\sf G}_i =: \sum_{i=1}^{r}{\sf G}_i$.

For every $j\in I$ we define the standard projection 
  $\pi_j\colon\prod_{i\in I}G_i \longrightarrow G_j$ by $\pi_j(g)=g(j)$,
and the standard inclusion 
  $\embfunc_j\colon G_j\longrightarrow \sum_{i\in I}G_i$
by the conditions 
  $(\embfunc_j(a))(j)=a$ and $(\embfunc_j(a))(i)=1_i$ for $i\neq j$
and $a\in G_j$.
Recall that $\pi_j$ and $\embfunc_j$ are group homomorphisms.

We set $\sum_{i\in I}D_i=\bigcup\{ \embfunc_i(D_i)\colon i\in I \}$.
For a finite set $I=\{ 1,\ldots r \}$ we write
$\sum_{i\in I}D_i = \sum_{i=1}^{r}D_i = D_1\uplus\ldots\uplus D_{r}$.

\begin{prop}
  If $D_i$ is a quasi difference set in ${\sf G}_i$ for every $i\in I$ then
  $\sum_{i\in I}D_i$ is a quasi difference set in $\sum_{i\in I}{\sf G}_i$.
\end{prop}
\begin{proof}
  We set $D = \sum_{i\in I}D_i$. The aim is to prove that $D$ satisfies \QDS.
  Let $g_1,g_2,g_3,g_4\in D$ and assume that
    $g_1\invers(g_2) = g_3\invers(g_4)$.
  Let $g_i\in \embfunc_{j_i}(D_{j_i})$. If $j_1=j_2$ then $\pi_j(g_1\invers(g_2)) = 1_j$
  for every $j\neq j_1$; 
  thus $\pi_j(g_3\invers(g_4))=1_j$, and thus $j_3 = j_4=j_1$.
  From assumption we infer that $g_1 = g_3$ and $g_2 = g_4$.
  If $j_1\neq j_2$, analogously, we come to $j_1 = j_3$ and $j_2 = j_4$.
  Then we obtain $g_1 = g_3$ and $\invers(g_2) = \invers(g_4)$,
  which yields our claim.
\end{proof}

\par\noindent
Let ${\goth D}_i = \DifSpace({\sf G}_i,D_i)$. We write 
  $$
    \textstyle{\sum_{i\in I}{\goth D}_i}:=
    \DifSpace(\textstyle{\sum_{i\in I}{\sf G}_i},\textstyle{\sum_{i\in I}D_i}).$$
For $I = \{1,\ldots,r\}$ we write also 
  $\sum_{i=1}^{r}{\goth D}_i := 
   \DifSpace({\sf G}_1,D_1)\oplus\ldots\oplus\DifSpace({\sf G}_r,D_r)$.
This terminology can be somewhat misleading (it is not true
that $\DifSpace({\sf G},D)$ determines the set $D$ in some standard way);
we hope this will not lead to misunderstanding, 
since in any case suitable quasi difference sets will be explicitly given.
\par
Let us examine two particular cases of the above construction.
First, let $\goth D=\DifSpace(C_k,{\{0,1\}})\oplus\DifSpace({\sf G}',D')$.
Set $D=\{0,1\}\uplus D'$ and denote ${\goth D}' = \DifSpace({\sf G}',D')$;
then ${\goth D} = \DifSpace(C_k\oplus {\sf G}',D)$.
Let $a = (i,a')\in C_k\times G'$. Then the points of $a\cdot D$ are of the form 
$(i,a'\cdot p)$ with $p\in D'$, and -- one point -- $(i+1,a')$.
Somewhat informally we can say that the line with the coordinates 
$[i,a']$ consists of the points $(i,p)$, where $(p)\inc [a']$ and one 
"new" point $(i+1,a')$.
\par
In other words, 
we have a function $f_i$ which assigns to every line of ${\goth D}'$ a point
of ${\goth D}'$ such that the lines of ${\goth D}$ are of the form
${i}\times l' \cup \{(i+1,f_i(l))\}$, where $l'$ is a line of ${\goth D}'$.
Thus we can consider $\goth D$ as a space ${\goth D}'$ $k$-times inscribed
cyclically into itself.
In the above construction, the function $f_i$ is defined by $f_i([a])=(a)$.
\par
Let $\goth D =\DifSpace({\sf G}_1,D_1)\oplus\DifSpace({\sf G}_2,D_2)$. 
The lines of $\goth D$ are of the form
$(a_1,a_2)+D_1\uplus D_2$, which, on the other hand, can be written
as
$\chain([a_1])\times\{(a_2)\}\cup\{(a_1)\}\times\chain([a_2])$. Recall, 
that the lines of the Segre product 
$\chain({\goth D})=\DifSpace({\sf G}_1,D_1)\otimes\DifSpace({\sf G}_2,D_2)$
(cf. \cite{Segre}) are the sets of one of two forms:
$\chain([a_1])\times\{(a_2)\}$ or $\{(a_1)\}\times\chain([a_2])$.
Therefore, the lines of $\goth D$ are unions of some pairs of the lines
of $\chain({\goth D})$.

Immediately from \ref{pr:concomp} we have the following
\begin{cor}
  Let $D_i$ be a quasi difference set in a group ${\sf G}_i$ 
  such that $\struct{D_i}_{{\sf G}_i} = G_i$ for $i\in I$.
  Then $\sum_{i\in I}D_i$ generates $\sum_{i\in I}{\sf G}_i$.
  Consequently, if every one of the structures 
  ${\goth D}_i=\DifSpace({\sf G}_i,D_i)$ is connected then 
  $\sum_{i\in I}{\goth D}_i$ is connected as well.
\end{cor}

\par
Let $J \subset I$; we extend the inclusions $\embfunc_i$ to the map
$\embfunc_J \colon \sum_{j \in J} G_j \longrightarrow \sum_{i \in I} G_i$
by the condition
\begin{equation}\label{wz:embed:gen}
  (\embfunc_J(a))(i) = \left\{ \begin{array}{ll}
                         1_i  & \text{for } i \in I \setminus J
                         \\
                         a(i) & \text{for } i \in J
                     \end{array} \right.
  \text{ for arbitrary } a \in \textstyle{\sum_{j \in J} G_j}.
\end{equation}
Let us denote
  ${\goth M} = \sum_{i \in I} {\goth D}_i$
and 
  ${\goth N} = \sum_{j \in J} {\goth D}_j$.
Next, let $c \in \sum_{i \in I \setminus J} G_i$.
Let 
  ${\goth M} \ciach{I \setminus J} c$   
be the substructure of $\goth M$ determined by the set 
  $\{ p \in \sum_{i \in I} G_i \colon p(i) = c(i), \; \forall \, i \in I \setminus J  \}$.
The following is immediate from definitions.
\begin{prop}\label{prop:partsofprod}
  The map $\tau_{\embfunc_{I \setminus J}(c)} \circ \embfunc_{J}$ 
  is an isomorphism of $\goth N$ and the structure ${\goth M} \ciach{I \setminus J} a$.
  Consequently, for any $c',c'' \in \sum_{i \in I \setminus J} G_i$ the map
  $\tau_{\embfunc_{I \setminus J}(c'')} \circ \invers({\tau_{\embfunc_{I \setminus J}(c')}})$
  is an isomorphism of 
  ${\goth M} \ciach{I \setminus J} c'$ and ${\goth M} \ciach{I \setminus J} c''$.
\end{prop}
A substructure of $\goth M$ of the form ${\goth M} \ciach{I \setminus J} c$ 
will be referred to as a {\em $J$-part of $\goth M$}; in fact, it is 
a Baer substructure of $\goth M$.

\par
With similar techniques we can prove
\begin{prop}\label{pr:associatprod}
  Let $J$ be a nonempty proper subset of $I$. Then
  \begin{eqnarray*}
  \textstyle{\sum_{i \in I} {\goth D}_i} 
  & \cong &
   \textstyle{\sum_{i \in J} {\goth D}_i} \oplus
   \textstyle{\sum_{i \in I \setminus J} {\goth D}_i}.
  \end{eqnarray*}
\end{prop}
\begin{prop}\label{pr:corelinprod}
  Let ${\goth D}_i = \DifSpace({\sf G}_i,D_i)$, where 
  $D_i$ is a quasi difference set in a group ${\sf G}_i$ for 
  $i\in I$.
  Assume that there is a pair of bijections
  $\varphi'_i,\varphi_i''\colon G_i\longrightarrow G_i$ 
  such that the pair $\varkappa_i=(\varkappa'_i,\varkappa''_i)$ of maps
  \begin{equation}
    \varkappa'_i\colon (a)\mapsto[\varphi'_i(a)],
    \;\;
    \varkappa''_i\colon [a]\mapsto(\varphi''_i(a))\text{ for }a\in G_i
  \end{equation}
  is a correlation of ${\goth D}_i$ for $i\in I$.
  Set 
    $\varphi'=\varphi'_1\times\ldots\times \varphi'_r $
    and $\varphi''=\varphi''_1\times\ldots\times \varphi''_r $.
  If $\varphi'_i = \varphi''_i$ for every $i\in I$
  (i.e. if $\varkappa_i$ are involutory)
  then the pair $\varkappa = (\varkappa',\varkappa'')$ of maps
  \begin{equation}
    \varkappa'\colon (a)\mapsto[\varphi'(a)],\;\;
    \varkappa''\colon [a]\mapsto(\varphi''(a))\text{ for }a\in
     \textstyle{\sum_{i \in I} G_i} 
  \end{equation}
  is an involutory correlation of 
    $\sum_{i\in I}{\goth D}_i$.
\end{prop}
\begin{proof}
  Let us note that, directly from the definition of $D=\sum_{i\in I} D_i$
  the following conditions  are equivalent
   for $a,b\in \sum_{i \in I} G_i$:
  \begin{enumerate}[a)]\itemsep-1pt
  \item
    $(a)\inc [b]$, and
  \item
    $(a_i)\inc[b_i]$ in ${\goth D}_i$ for some $i\in I$ and $a_j=b_j$ for 
    $i\neq j\in I$.
  \end{enumerate}
  Indeed, $a\in b\cdot D$ iff $a=b\cdot d$ for some 
  $d\in \bigcup\{ \embfunc_i(D_i)\colon i\in I \}$
  i.e. iff $a_i\in b_i\cdot D_i$ and $a_j=b_j$ for all $j\neq i$.
  Therefore,
  \\
  \centerline{
  $\varkappa''([b]) = (\varphi''(b))\inc[\varphi'(a)]
  = \varkappa'((a))$}
  iff the following holds:
  \\
  \centerline{
  $\varkappa''_i([b_i]) = (\varphi''_i(b_i))\inc[\varphi'(a_i)]=\varkappa'_i((a_i))$ 
  and $\varphi''_j(b_j)=\varphi'_j(a_j)$ for $j\neq i$.}
  Now the claim is evident.
\end{proof}

Note that, in particular, if we assume in \ref{pr:corelinprod} that every $\varkappa_i$
is the standard correlation ($\varphi'_i(a) = - a$, cf. \ref{pr:corelindif}),
then $\varkappa$ is the standard correlation as well.
\begin{prop}\label{pr:autinsegre}    
  Let ${\goth D}_i=\DifSpace({\sf G}_i,D_i)$ and
  $f_i=(f'_i,f''_i)$ be bijections of $G_i$ such that 
  $f'_i\colon (a)\longmapsto(f'_i(a))$ and
  $f''_i\colon [a]\longmapsto[f''_i(a)]$ yields a collineation 
  of ${\goth D}_i$ for $i\in I$, and let ${\goth D}=\sum_{i\in I}{\goth D}_i$.
  We set $F'=\prod_{i\in I}f'_i$, 
  $F''=\prod_{i\in I}f''_i$, and $F=(F',F'')$.
  Then the pair $F$ is a collineation of $\goth D$ iff $f'_i=f''_i$
  for every $i\in I$.
\end{prop}
\begin{proof}
  Let $p\in \sum_{i\in I} G_i$, take arbitrary $i\in I$.
  Let the points $g_1,g_2$ of the form
  \begin{eqnarray}
    g_s(j)=\left\{ \begin{array}{ll}
    p(j) & \textrm{for $j\neq i$}\\
    p(j)+ d^s & \textrm{for $j=i$}
    \end{array} \right.
  \end{eqnarray}
  lie on the line $[p]$. Consequently, their images are $F(g_s)$ with
  \begin{eqnarray}\label{imofprod}
    F(g_s)(j)=\left\{ \begin{array}{ll}
    f'(p_j) & \textrm{for $j\neq i$}\\
    f'(p_j+d^s) & \textrm{for $j=i$} ,
    \end{array} \right.
  \end{eqnarray}
  for $d^1,d^2\in D_i$.
  Besides, from the assumption, the points $F(g_1),F(g_2)$ lie on $[f''(p)]$ 
  and $f''(p)(j)=f''_j(p_j)$ as well. Furthermore, if $F(g_s)\in [f''(p)]$ then 
  according to \eqref{imofprod} we obtain $f''(p)(j)=f'_j(p_j)$. 
  Finally, we get that
  $F$ is a collineation of $\goth D$ iff $[f''_i(a)] = f''_i([a]) = [f'_i(a)]$
  for every $a\in G_i$.
\end{proof}
Obviously, the pair $(f'_i,f''_i)$, where $f'_i=f''_i={\tau}_{a_i}$ and $a_i \in G_i$,
is a collineation of ${\goth D}_i$;
therefore, the pair 
$(\prod_{i \in I} f'_i , \prod_{i \in I} f''_i)$
is an automorphism of $\goth D$.
But this is a rather trivial result, as
$\prod_{i \in I} \tau_{a_i} = \tau_{a}$.
We have also some automorphisms of another type:
\begin {lem}\label{lem:syminsegre}
  Let $\beta \in S_n$ and let the map $h \colon G^n \longrightarrow G^n$
  be defined by the condition
  $h((x_1,\ldots,x_n))=(x_{\beta(1)},\ldots,x_{\beta(n)})$.
  Then the pair $F=(f'_i,f''_i)=(h,h)$
  is a collineation of 
    $$\underbrace{\DifSpace({\sf G},D)\oplus \DifSpace({\sf G},D)\oplus \ldots
       \DifSpace({\sf G},D)}_{n \; \mathrm{times}} = \DifSpace({\sf G}^n,D^n).$$
\end{lem}
\begin {proof}
  Take a point $(x) = (x_1,\ldots,x_n)$ and a line $[y] = [y_1,\ldots,y_n]$, such that
  $(x)\inc [y]$. Then, there exists $i\in I$ with $x_i = y_i+d^i$, $d^i\in D$ and 
  $x_j = y_j$ for all $j\neq i$. Images of $x$ and $y$ under $h$ satisfy analogous condition,
  with $i$ replaced by $\beta(i)$
  and thus $(h(x)) \inc [h(y)]$.
\end{proof}



\subsection{Cyclic multiplying} 
In this section we shall be mainly concerned with structures of the
form $\DifSpace({\sf G},{\cal D}_r)$, determined by
${\sf G} = C_{n_1}\oplus\ldots\oplus C_{n_{r}}$ and 
${\cal D}_r = \{ e_0,e_1,\ldots,e_r \}$,
where $e_0=(0,0,\ldots,0)$ and 
$(e_i)_j=0$ for $i\neq j$, $(e_i)_i=1$. Such a set ${\cal D}_r$ will be called
{\em canonical}. 
One can observe that 
  $\DifSpace({\sf G},{\cal D}_r) \cong \sum_{i=1}^r\DifSpace(C_{n_i},\{ 0,1 \})$
and, on the other hand 
  $\DifSpace({\sf G},{\cal D}_r) \cong 
    \DifSpace(C_{n_1},\{ 0,1 \})\oplus
    \DifSpace(C_{n_2}\oplus\ldots\oplus C_{n_r},{\cal D}_{r-1})$.
Thus defining a structure $\DifSpace({\sf G},{\cal D}_r)$ we generalize
a construction of cyclically inscribed polygons.
Figure \ref{fig:multpap0} illustrates the structure
$\DifSpace(C_3\oplus C_3\oplus C_3,{\cal D}_3)$ which, on the 
other hand can be considered as consisting of three copies of Pappus configuration 
cyclically inscribed.

\begin{figure}[!h]
    \begin{center}
    \includegraphics[scale=0.7]{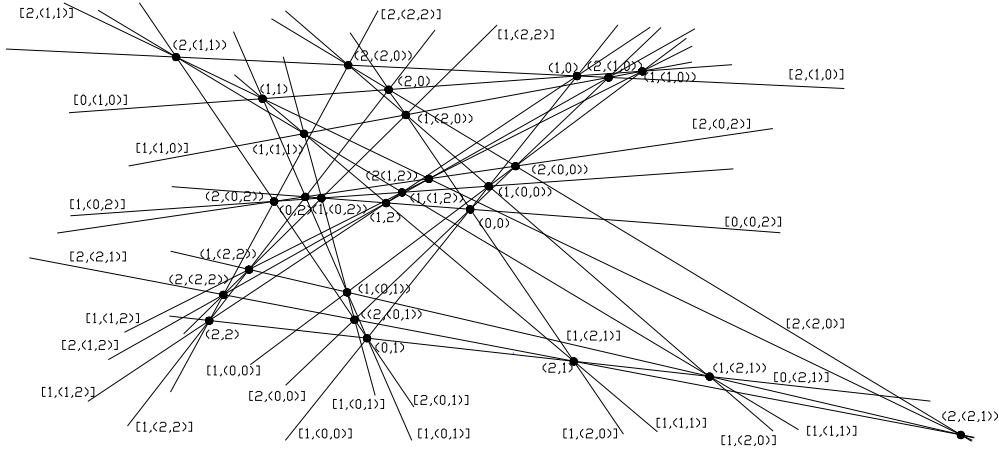}
    \end{center}
\caption{Multiplied Pappus configuration}
\label{fig:multpap0}
\end{figure}


\begin{lem}\label{lem:permtocollin}
  Let $\goth M= \DifSpace((C_k)^n, {\cal D}_n)$, where ${\cal D}_n$ is the canonical
  quasi difference set in abelian group $(C_k)^n$.
  \begin{sentences}
  \item\label{permtocollin:cas1}
    For every permutation $\alpha$ of the set $\{0,\dots ,n\}$, such that 
    $\alpha(0)=0$, there
    exists a collineation $f=(f',f'')$ of the structure $\goth M$ such that 
      $f'(e_0)=e_0$ and $f''(-e_i)=-e_{\alpha(i)}$ for $i= \{0,\dots ,n\}$.
  \item\label{permtocollin:cas2}
    For every transposition $\alpha$ of the set $\{0,\dots ,n\}$, such that 
    $\alpha(0)=s\neq 0$, there
    exists a collineation $f=(f',f'')$ of the structure $\goth M$ such that 
      $f'(e_0)=e_0$ and $f''(-e_i)=-e_{\alpha(i)}$ for $i= \{0,\dots ,n\}$.  
   \item \label{permcol}\label{permtocollin:cas3}
    For every permutation $\alpha$ of the set $\{0,\dots ,n\}$ there
    exists a collineation $f=(f',f'')$ of the structure $\goth M$ such that 
      $f'(e_0)=e_0$ and $f''(-e_i)=-e_{\alpha(i)}$ for $i= \{0,\dots ,n\}$.  
  \item \label{eftauef}\label{permtocollin:cas4}
    If $\alpha$ is a permutation of the set $\{0,\dots ,n\}$ and $f=(f',f'')$ is a
    collineation of the structure $\goth M$ such that 
      $f'(e_0)=e_0$, $f''(-e_i)=-e_{\alpha(i)}$ for $i= \{0,\dots ,n\}$, then
      $f' \tau_v (f')^{-1}= \tau_{f'(v)}$.
  \end{sentences}
\end{lem}
\begin{proof}
\eqref{permtocollin:cas1}:
  Let us define a function $f': G \longrightarrow G$ by following formula:
  \begin{equation}\label{permzerofix}
     f'(x_1,\dots,x_n)=(x_{\alpha(1)},\dots,x_{\alpha(n)}).
  \end{equation}
  Then $f'\in \Aut({\sf G})$ and $f'({\cal D}_n)={\cal D}_n$, thus $f'$ determines an 
  automorphism of $\goth M$. 
  In view of \ref{rem:whenautline} we get 
    $f''(y_1,\dots,y_n)=(y_{\alpha(1)},\dots,y_{\alpha(n)})$ 
  so,
    $f''(-e_i)=-e_{\alpha(i)}$.

\eqref{permtocollin:cas2}:
  We define a function $f': G \longmapsto G$ by formula:
  \begin{equation}\label{permtransp}
     f'(x_1,\dots,x_n)=(x_1,\dots,x_{s-1},-\sum_{i=1}^n x_i,x_{s+1},\dots,x_n),
  \end{equation}
  then $f'\in \Aut({\sf G})$. It is easy to notice that $f'({\cal D}_n)=-e_s+{\cal D}_n$. 
  From \ref{rem:whenautline}, $f'$ gives a collineation and 
    $f''(y_1,\dots,y_n)=(y_1,\dots,y_{s-1},-\sum_{i=1}^n y_i-1,y_{s+1},\dots,y_n)$,
  and thus
    $f''(-e_i) = -e_{\alpha(i)}$. 

\eqref{permtocollin:cas3} follows immediately from \eqref{permtocollin:cas1} 
and \eqref{permtocollin:cas2}.

\eqref{permtocollin:cas4}:
  Let $\alpha$ be a permutation such that $\alpha(0)=0$, then 
    $f'$ is given by \eqref{permzerofix}
  and 
    $(f')^{-1}(x_1,\dots,x_n)=(x_{\alpha^{-1}(1)},\dots,x_{\alpha^{-1}(n)})$.
  Therefore 
    $f' \tau_v (f')^{-1}(x_1,\dots,x_n)= (x_1,\dots,x_n) + f'(v)=
       \tau_{f'(v)}(x_1,\dots,x_n)$.

  If $\alpha$ is a transposition then $f'$ is given by \eqref{permtransp}, thus 
  $f'=(f')^{-1}$ and, analogously, after simple calculation we get 
  our claim.
\end{proof}
\begin{lem}\label{lem:multpapneb}
  Let ${\goth M} = \DifSpace((C_3)^n, {\cal D}_n)$ and 
  $\theta = (0,\ldots,0)$.
  For every  $k,i,j \in \{ 0,\ldots,n \}$ with $k\neq j,i$  
  there is exactly one point $q_{k,i;j}$ such that 
  $\theta \neq q_{k,i;j} \inc [-e_i]$ 
  and $q_{k,i;j}$ is collinear with $(-e_k + e_j)\inc [-e_k]$.
  We have
  \begin{equation}\label{wz:multpapneb}
    q_{k,i;j} = \left\{
              \begin{array}{ll}
                (-e_i + e_j)   & \text{ if }\; i \neq j
                \\
                (-e_i + e_k)   & \text{ if }\; i = j
              \end{array}
              \right.
  \end{equation}
  Consequently, for every point $o$ of $\goth M$, any two distinct lines $L_1,L_2$
  through $o$ and every point $p$ with  $o \neq p \inc L_1$ there is the unique point
  $q$ such that $o\neq q \inc L_2$ and $p \collin q$.
\end{lem}
\begin{proof}
  Let us consider, first, the case $k=0$.
  Clearly, $(e_j),(-e_i+e_j) \inc [-e_i+e_j]$ and $o \neq (-e_i+e_j) \inc [-e_i]$
  for $i\neq j$ (cf. \ref{lem:veblen}).
  If $(e_j) \collin (-e_i+e_k) = q \inc [-e_i]$ and $i\neq j$ then, by 
  \ref{lem:analchain}\eqref{analchain:cas2} we have 
  $u := e_j + e_i - e_k \in {\cal D}_n - {\cal D}_n$.
  Then $u$ is a combination of at most two elements of ${\cal D}_n$
  and thus  
  $i=k$ or $j=k$.
  If $i=k$ then $q = \theta$; if $j=k$ we obtain the claim.

  Further, since $2e_i = -e_i$ we have $(e_i),(-e_i) \inc [e_i]$ and, clearly,
  $(-e_i) \inc [-e_i]$.
  The requirement $(e_i)\collin (-e_i+e_k)\inc [-e_i]$ in view of
  \ref{lem:analchain}\eqref{analchain:cas2} gives $e_i+e_k\in{\cal D}_n-{\cal D}_n$,
  which has only two solutions: $e_k = e_0 = \theta$ and $e_k = e_i$.
  This proves the formula \eqref{wz:multpapneb} for $k=0$.

  For arbitrary $k$ we can see that 
  $(-e_k + e_j),(-e_i + e_j) \inc [-e_k - e_i + e_j]$ and 
  $\theta \neq (-e_i + e_j) \inc [-e_i]$ for $i\neq j$, and
  $(-e_k + e_i),(-e_i + e_k) \inc [e_k + e_i]$ and
  $\theta \neq (-e_i + e_k) \inc [-e_i]$.
  Since the stabilizer of the point $\theta$ in the automorphism group of $\goth M$
  acts transitively on lines through $\theta$ (cf. \ref{lem:permtocollin})
  we proved \eqref{wz:multpapneb}.

  To close the proof it suffices to recall that $\goth M$ is homogeneous:
  translations form a transitive group of automorphisms of $\goth M$.
\end{proof}
\begin{prop}\label{cor:semidirprod}
  Let $k > 3$.
  Under the notation from \ref{lem:permtocollin} 
  the group $\otocz({\Aut({\goth M})},o)$ with
  $o=(0,\dots,0)$ is isomorphic to $S_{n+1}$, while 
  $\Aut({\goth M})\cong S_{n+1} \ltimes (C_k)^n$.
\end{prop}       
\begin{proof}
  To every collineation $f=(f',f'')$, which fixes $o$, there is 
  a permutation $\alpha= \alpha_f$ uniquely assigned such that $f''(-e_i)=-e_{\alpha(i)}$. 
  Clearly, from \ref{cor:fixautisomor} the map 
    $f \mappedby{{\xi}}   \alpha_{f}$ 
  is a group monomorphism. 
  From \eqref{permcol} of \ref{lem:permtocollin}, $\xi$ is an epimorphism so, 
  $\xi$ is an isomorphism of $S_{n+1}$ and 
  $\otocz({\Aut({\goth M})},o)$. 
  We know that 
    $\text{Tr}({\sf G})\cong {\sf G}$ and 
    $\Aut({\goth M}) = \text{Tr}\big((C_k)^n\big)\circ \otocz({\Aut({\goth M})},o)$. 
  Based on \ref{lem:permtocollin}\eqref{eftauef}  we have 
   $(\tau_u f_{\alpha})(\tau_v f_{\beta})= \tau_u \tau_{f_{\alpha}(v)} f_{\alpha \beta}$, 
  which yields our claim.
\end{proof}

Then, we shall pay some attention to the more general case.
Namely, we describe the neighborhood of a point $q$ in the configuration
of the form
${\goth M} = \DifSpace(C_k,{\cal D}_2)\oplus \DifSpace({\sf G},D)$,
i.e. 
$\DifSpace({C_k\oplus {\sf G}},{\{ 0,1 \} \uplus D})$,
where $D$ is a quasi difference set in an abelian group $\sf G$.
Since $\goth M$ has a point-transitive automorphisms group, 
without loss of generality we can assume 
that $q = (0,\theta)$, where $\theta$ is the zero of $\sf G$.
Immediately from definitions we calculate the following
\begin{lem}\label{lem:papinscneighb}
  Let $D = \{ d_0,\ldots,d_n  \}$ 
  be a quasi difference set in an abelian group ${\sf G} = \struct{G,+,\theta}$, 
  and let $q = (0,\theta) \in C_k \times G$.
  Set ${\goth M} = \DifSpace(C_k,{\cal D}_2)\oplus \DifSpace({\sf G},D)$.
  The lines of $\goth M$ through $q$ are the following:
  \begin{enumerate}[1:]
  \item
    $l_i = [0,- d_i]$ for $i = 0,\ldots,n$.
    \par
    Each line $[0,-d_i]$ contains $q$ and the following points:
    \begin{enumerate}[a:\quad]
    \item
      $q_{i,j} = (0,-d_i + d_j)$ for $j=0,\ldots,i$, $i\neq j$;
    \item
      $p'_i = (1,-d_i)$.
    \end{enumerate}
  \item
    $l'' = [-1,\theta] = [k-1,\theta]$.
    \par
    Its points are $q$ and the following
    \begin{enumerate}[a:\quad]\setcounter{enumii}{2}
    \item
      $p''_i = (-1,d_i)$ for $i=0,\ldots,n$.
    \end{enumerate}
  \end{enumerate}
  Then the points $q_{i,j}$ form a substructure isomorphic under the map 
  $(0,a) \longmapsto (a)$ to the neighborhood of $\theta$ in $\DifSpace({\sf G},D)$.
  Moreover, the following additional connecting lines appear:
  \begin{enumerate}[1:]\setcounter{enumi}{2}
  \item\label{extramult:cas3}
    For every $i,j=0,\ldots,n$, $i\neq j$
    the line $l''_{i,j} = [-1,-d_j + d_i]$
    joins $p''_i = (-1,d_i) \inc [-1,\theta] = l''$ 
    with $q_{j,i} = (0,-d_j + d_i) \inc [0,-d_j] = l_j$.
  \item\label{extramult:cas4}
    For every $i,j$ as above, 
    the line $l'_{i,j} = l'_{j,i} = [1,-(d_i + d_j)]$
    joins $p'_i = (1,-d_i) \inc [0,-d_i] = l_i$
    with $p'_j = (1,-d_j) \inc [0,-d_j] = l_j$.
  \end{enumerate}
  The lines listed above are pairwise distinct.
  \begin{sentences}
  \item
    If $k > 3$, then no other connecting line appears.
  \item
    Let $k = 3$. Then $-1 = 2$ holds in $C_k$, and then another connections
    are associated with triples $(d_i,d_j,d_r)\in D^3$  satisfying
    \begin{equation}\label{wz:extrajoin}
      d_i + d_j + d_r = \theta.
    \end{equation}
    Namely, let \eqref{wz:extrajoin} be satisfied. 
    Evidently, $-(d_i+d_j) = d_r$.
    \begin{enumerate}[1:]\setcounter{enumi}{4}
    \item
      The line $l'_{i,j} = [1,-(d_j+d_i)] = [1,d_r]$
      connects $p''_r = (-1,d_r) \inc [-1,\theta] = l''$
      with $p'_j = (1,-d_j) \inc [0,-d_j] = l_j$.
    \item
      If, moreover, $j\neq i$, then the above line passes through $p'_i$ as well so,
      it coincides with the line defined in \eqref{extramult:cas4}.
    \end{enumerate}
  \end{sentences}
\end{lem}
\noindent
The cases presented above are illustrated in Figures
\ref{fig:multneighb:gen} and \ref{fig:multneighb:gen3}.

\begin{figure}[!h]
    \begin{center}
    \includegraphics[scale=0.6]{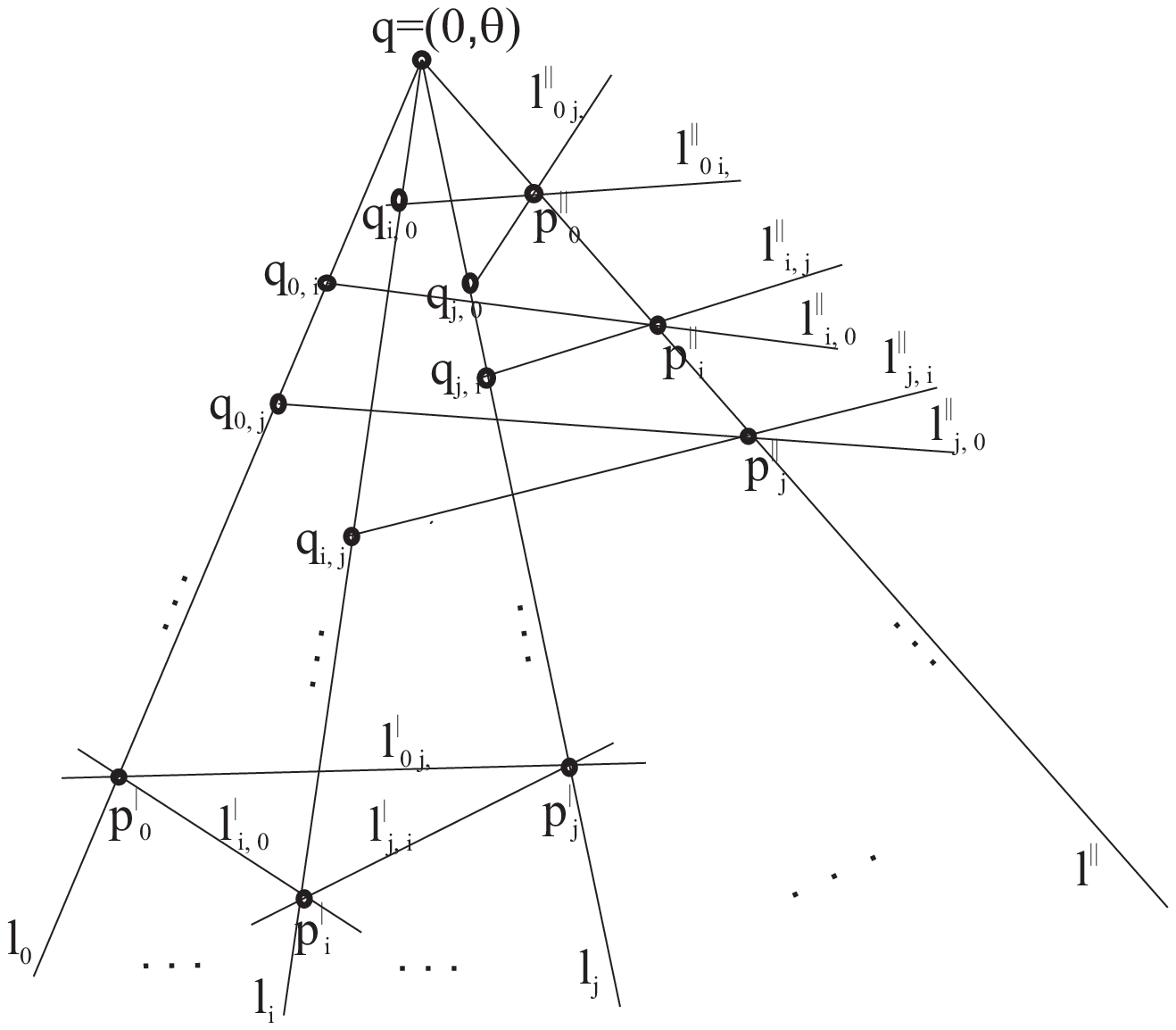}
    \end{center}
\caption{The neighborhood of the point $(0,\theta)$ in a configuration %
$\DifSpace(C_k,{\cal D}_2)\oplus \DifSpace({\sf G},D)$}
\label{fig:multneighb:gen}
\end{figure} 

\begin{figure}[!h]
    \begin{center}
    \includegraphics[scale=0.6]{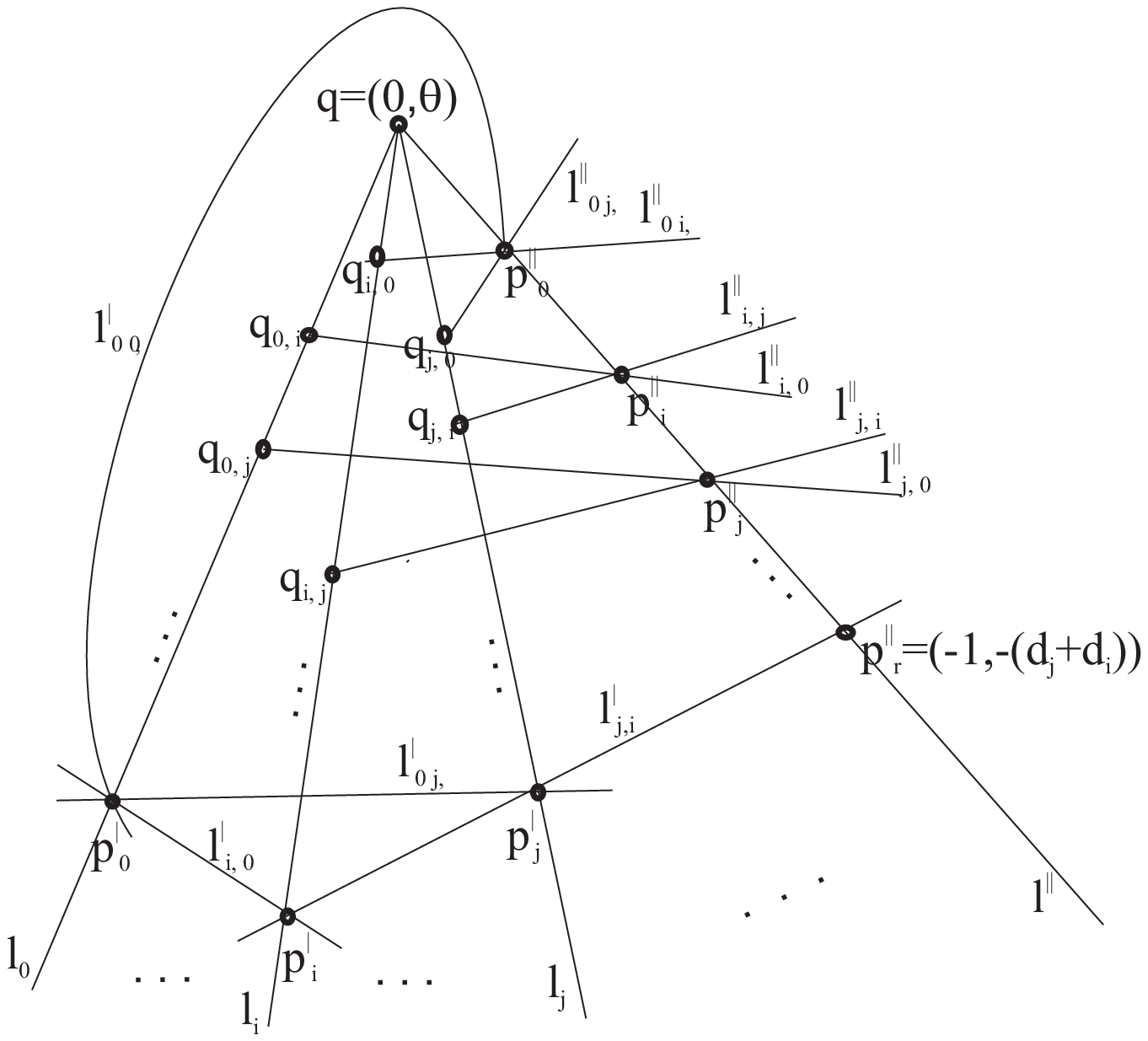}
    \end{center}
\caption{The neighborhood of the point $(0,\theta)$ in a configuration %
$\DifSpace(C_k,{\cal D}_2)\oplus \DifSpace({\sf G},D)$ for $k=3$}
\label{fig:multneighb:gen3}
\end{figure} 

In the case of configurations determined by quasi difference sets we 
can apply also some other techniques generalizing those of \cite{petel}.
Let ${\goth D}_0=\DifSpace({\sf G}_0,D)$ for a quasi difference set in a 
group ${\sf G}_0$, let $k$ be an integer, and
${\goth D}=\DifSpace(C_k\oplus{\sf G}_0,D\uplus\{0,1\})$.
Then, let $f=(f',f'')$ be a collineation of ${\goth D}_0$.
Recall that a collineation of $\DifSpace(C_k,\{0,1\})$ is simply an element of
the dihedral group $D_k$, i.e. it is any map 
$\alpha_{\varepsilon,q} \colon i\mapsto \varepsilon i + q$, where 
$\varepsilon\in\{1,-1\}$.
Proposition \ref{pr:autinsegre} determines all the automorphisms of $\goth D$ 
of the form $(i,a)\mapsto(\alpha_{\varepsilon,q}(i),f'(a))$.
Still, in this case we should look for automorphisms defined with  more 
complicated formulas.

\begin{prop}\label{pr:extautinmult}
  Let ${\goth D}$ be defined as above and $f=(f',f'')\in \Aut({\goth D}_0)$.
  The following conditions are equivalent:
  \begin{sentences}
  \item\label{extautinmult:cas1}
    There is a collineation $\varphi=(\varphi',\varphi'')$ of $\goth D$
    such that $\varphi'((0,a))=(0,f'(a))$ and $\varphi''([0,b])=[0,f''(b)]$.
  \item\label{extautinmult:cas2}
    There is a sequence $f_i$ ($i\in C_k$) of collineations of ${\goth D}_0$
    defined recursively by the formulas:
    $f_0=f$ and $f'_{i+1}=f''_{i}$, where $f_i=(f'_i,f''_i)$.
  \end{sentences}
  In the case \eqref{extautinmult:cas2} we have 
  $\varphi'((i,a))=(i,f'_i(a))$ and $\varphi''([i,b])=[i,f''_i(b)]$.
\end{prop}
\begin{proof}
  It suffices to note that if \eqref{extautinmult:cas1} holds then 
  $(1,a)\inc[0,a]$ for every $a\in G_0$, which gives, necessarily, 
  $\varphi'((1,a))\inc\varphi''([0,a])=[0,f''(a)]$ and thus 
  $\varphi'((1,a))=(1,f''(a))$.
  Therefore, $f''$ (as a transformation of points)
  must determine a collineation of ${\goth D}_0$.
\end{proof}


\section{Elementary properties}

Lemma \ref{lem:veblen} enables us to discuss some elementary axiomatic 
properties of the structures $\DifSpace({\sf G},D)$.
Let $D$ be a quasi difference set in a commutative group ${\sf G}$.

\begin{prop}\label{pr:veblen}
  Under assumption \eqref{ass:veblen} 
  the structure $\DifSpace({\sf G},D)$ is Veblenian.  
\end{prop}
\begin{proof}
  Set ${\goth D}=\DifSpace({\sf G},D)$.
  Let $(a)$ be a point of $\goth D$ and $[b_1],[b_2]$ be two distinct lines of
  $\goth D$ through $(a)$.
  From \eqref{wz:analchain1}, $b_i=a\invers(d_i)$ for some $d_1,d_2\in D$ 
  with $d_1\neq d_2$.
  Consider any two lines $[g_1],[g_2]$ which cross $[b_1]$ and $[b_2]$
  and do not pass through $(a)$.
  From \ref{lem:veblen}(iii) 
  $g_i = a\invers(d_1)\invers(d_2)d'_i$ for some $d'_1,d'_2\in D$
  and then \ref{lem:veblen}(iv) yields 
  that the lines $[g_1],[g_2]$ intersect each other,
  which proves our claim.
\end{proof}

\begin{prop}\label{pr:desargues}
  Under assumption  \eqref{ass:veblen} 
  the structure $\DifSpace({\sf G},D)$ is Desarguesian.  
\end{prop}
\begin{proof}
  Let $[b_i]=[a\invers(d_i)]$ with $d_i\in D$ be three lines through a point $(a)$, and 
  let $(p'_i),(p''_i)$ be two triples of points such that 
  $(p'_i),(p''_i)\inc[b_i]$ and $p'_i,p''_i\neq a$
  for $i=1,2,3$,
  and $((p'_1),(p'_2),(p'_3))$, $((p''_1),(p''_2),(p''_3))$ are triangles.
  With \ref{lem:veblen} we get for $\{i,j,k\}=\{1,2,3\}$ that
    $$[g'_k] := \LineOn(p'_i,p'_j) = [a\invers(d_i)\invers(d_j)d']
    \;\text{  and  }\; 
    [g''_k] := \LineOn(p''_i,p''_j) = [a\invers(d_i)\invers(d_j)d'']$$
  for some $d',d''\in D$.
  From \ref{lem:veblen} we get
    $$(q_k) := \PointOf(g'_k,g''_k) = (a\invers(d_i)\invers(d_j)d'd'').$$
  To close the proof it suffices to observe that
  $(q_1),(q_2),(q_3)\inc[a\invers(d_1)\invers(d_2)\invers(d_3)d'd'']$.
\end{proof}

Let us note that under assumptions of \ref{lem:veblen} the structure 
$\DifSpace({\sf G},D)$ does not contain any Pappus configuration.
Indeed, (cf. e.g. \cite{mald} or \cite{petel})
the Pappus configuration can be considered as 
$\DifSpace(C_3\oplus C_3,D_0)$, where $D_0=\{(0,0),(0,1),(1,0)\}$,
see Figure \ref{fig:pap0}. Then 
$(0,0),(1,0),(0,1)\inc[0,0]$;\space $(0,0),(1,2),(0,2)\inc[0,2]$; $(1,0)\collin(1,2)$, and 
$(0,1)\collin(0,2)$, which contradicts \ref{lem:veblen}(ii).

\begin{figure}[!h]
    \begin{center}
    \includegraphics[scale=0.6]{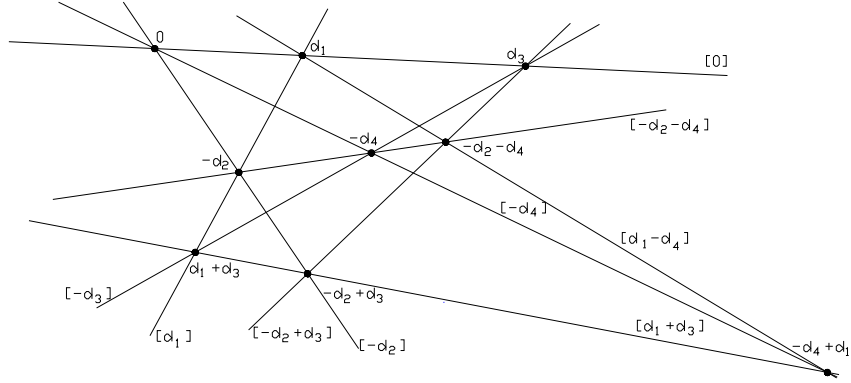}
    \end{center}
\caption{Pappus configuration}
\label{fig:pap0}
\end{figure} 

\begin{lem}\label{lem:pappus}
  Let $D$ be a quasi difference set in an abelian group $\sf G$ such that
  $1\in D$. Assume that there are $d_1,d_2,d_3,d_4\in D\setminus\{1\}$ with 
  $d_1 \neq d_3$, $d_1^2=\invers(d_2)$, and $d_3^2=\invers(d_4)$.
  Then ${\goth D}=\DifSpace({\sf G},D)$ contains Pappus configurations.
\end{lem}
\begin{proof}
  From the assumptions we get $d_2 \neq d_4$ as well.
  Note that incidences indicated in the following table hold in $\goth D$:
  
\begin{footnotesize}
\incsign{$\times$}
$$\left[
  \begin{inctab}{r||c|c|c|c|c|c|c|c|c}
  \strut & 
    $(1)$ & $(d_1)$ & $(d_3)$ & 
    $(\invers(d_2))$ & $(\invers(d_4))$ & $(\invers(d_2)\invers(d_4))$ & 
    $d_1d_3$ & $\invers(d_2)d_3$ & $\invers(d_4)d_1$
  \\
  \hline\hline
  \makeincline
  {$[1]$} 
      +   +   +   -   -   -   -   -   -;
  {$[\invers(d_2)\invers(d_4)]$}
      -   -   -   +   +   +   -   -   -;
  {$[d_1d_3]$}
      -   -   -   -   -   -   +   +   +;
  {$[\invers(d_2)]$}
      +   -   -   +   -   -   -   +   -;
  {$[\invers(d_4)]$}
      +   -   -   -   +   -   -   -   +;
  {$[d_1]$}
      -   +   -   +   -   -   +   -   -;
  {$[d_1\invers(d_4)]$}
      -   +   -   -   -   +   -   -   +;
  {$[d_3]$}
      -   -   +   -   +   -   +   -   -;
  {$[d_3\invers(d_2)]$}
      -   -   +   -   -   +   -   +   -;.
  \end{inctab}
\right]$$
\end{footnotesize}

Then the map defined for the points by
$$
\begin{tabular}{ccccccccc}
    $(0,0)$ & $(1,0)$ & $(0,1)$ & 
    $(1,2)$ & $(2,1)$ & $(1,1)$ & 
    $(2,2)$ & $(0,2)$ & $(2,0)$
\\ \hline
$(1)$ & $(d_1)$ & $(d_3)$ & 
    $(\invers(d_2))$ & $(\invers(d_4))$ & $(\invers(d_2)\invers(d_4))$ & 
    $d_1d_3$ & $\invers(d_2)d_3$ & $\invers(d_4)d_1$
\end{tabular}$$
and for the lines by 
$$
\begin{tabular}{ccccccccc}
$[0,0]$ & $[1,1]$ & $[2,2]$ & 
    $[0,2]$ & $[2,0]$ & $[1,2]$ & 
    $[1,0]$ & $[2,1]$ & $[0,1]$
\\ \hline
$[1]$ & $[\invers(d_2)\invers(d_4)]$ & $[d_1d_3]$ & 
    $[\invers(d_2)]$ & $[\invers(d_4)]$ & $[d_1]$ & 
    $[d_1\invers(d_4)]$ & $[d_3]$ & $[d_3\invers(d_2)]$
\end{tabular}$$
embeds the Pappus configuration into the structure $\goth D$.
\end{proof}

Note that if ${\cal D}_n$ is the canonical quasi difference set in
the abelian group $(C_k)^n$ then the structure
  $\goth M = \DifSpace((C_k)^n, {\cal D}_n)$
satisfies \eqref{ass:veblen}  iff $k>3$. 
Then, as a consequence of \ref{pr:veblen} and \ref{pr:desargues}
we get
\begin{cor}\label{cor:vebdes}
  Let $k > 3$.
  Then the structure $\goth M = \DifSpace((C_k)^n, {\cal D}_n)$ is Veblenian and
  Desarguesian.
\end{cor}       
\begin{prop}\label{prop:veb3}
  For $k=3$ the structure $\goth M = \DifSpace((C_k)^n, {\cal D}_n)$ is 
 not Veblenian. 
\end{prop}
\begin{proof}
  Let $(e_0)$, $(-e_1+e_2)$, $(-e_2+e_1)$ be a triangle in $\goth M$.
  We take $L_i=[-e_i]$ for $i=1,2$, and $K_1=[e_1+e_2]$, $K_2=[-e_1-e_2+e_3]$.
  Then $(e_0)\inc L_1,L_2$, and $K_1$ crosses $L_1$ in $(-e_1+e_2)$ and $L_2$ in
  $(-e_2+e_1)$. Furthermore, the lines $K_2$, $L_1$ meet in $(-e_1+e_3)$, and
  $(-e_1+e_3)$ is the common point of $K_2$, $L_2$.
 \par
  Suppose that $K_1\cap K_2\neq \{\emptyset\}$; then 
  $(e_1+e_2)+e_t = (-e_1-e_2+e_3)+e_s$ for some $s,t=\{1,\ldots,n\}$, $s\neq t$,
  which implies a contradiction: $e_1+e_2+e_3+e_s = e_t$. 
\end{proof} 
\begin{prop}\label{prop:des3}
  For $k=3$ the structure $\goth M = \DifSpace((C_k)^n, {\cal D}_n)$ is 
  Desarguesian. 
\end{prop}
\begin{proof}  
  Let $k = 3$.
  We gather some simple facts, useful in further part of the proof.
  The structure $\goth M $ has the following properties:
  \begin{enumerate}[1)]\itemsep-2pt
   \item 
     all the lines through the point $(e_0)$ are of the form 
     $[-e_i]$, where $i=1,\ldots,n$;\label{linethrough}
   \item 
     $(-e_i+e_s),(-e_j+e_s)\inc [-e_i-e_j+e_s]$, where 
     $j,s=1,\ldots,n$, $s\neq i,j$;\label{cross} 
   \item 
     $(-e_i+e_j),(-e_j+e_i)\inc [e_i+e_j]$;\label{cross3}
   \item 
     $[-e_i-e_j+e_s]$ crosses $[-e_i-e_j+e_t]$ in $(-e_i-e_j+e_s+e_t)$, where
        $t=1,\ldots,n$, $t\neq i,j$;\label{posvebindes}
   \item 
     $[-e_i-e_j+e_s]$ does not cross $[e_i+e_j]$;\label{negvebindes} 
   \item 
     there is no other line which crosses both $[-e_i]$ and $[-e_j]$ except
        $[-e_i-e_j+e_s]$ or $[e_i+e_j]$.\label{onlylines}
  \end{enumerate}
 %
 Indeed:
 \ref{linethrough}), \ref{cross}), \ref{cross3}), \ref{posvebindes})
 -- follow immediately from \ref{lem:veblen}, \ref{lem:analchain}.
 
 Ad \ref{negvebindes}): If $[-e_i-e_j+e_s]$ intersects $[e_i+e_j]$ then for some
 $r,t$ we have $e_i+e_j+e_s = e_r-e_t$ -- a contradiction.
 
 \ref{onlylines})  follows from \ref{lem:multpapneb}.

\par
  Without loss of generality we can assume that $o = (e_0)$ is the perspective center of two 
 triangles $T_1$, $T_2$, inscribed into
 three lines $L_1,L_2,L_3$ of $\goth M$ such that the corresponding pairs of their
 sides intersect each other. From \ref{linethrough}), the $L_r$ for $r = 1,2,3$ 
 are of the form $L_r = [-e_{i_r}]$.  From \ref{cross}), \ref{cross3}), and
 \ref{onlylines}) we get that  the sides of these triangles can be the 
 lines $[-e_i-e_j+e_s]$ or $[e_i+e_j]$.  
 On the other hand, none of the sides of the 
 triangles inscribed into the $L_r$ can have the form $[e_i+e_j]$, since, 
 as a consequence of \ref{negvebindes}),
 the line $[e_i+e_j]$ (joining two points: one on $L_i$, and the second on $L_j$)
 does not intersect any other line 
 which crosses both $L_i$ and $L_j$, and which should be a side of the
 second corresponding triangle.
 Then the sides are lines of the type $[-e_i-e_j+e_s]$, which pairwise intersect 
 each other respectively, in view of \ref{posvebindes}). 
 Consequently, we find that the triangles $T_1$, $T_2$ 
 have as their vertices the  points 
 \newline\centerline{
 $T_1: \; (-e_{i_1} + e_s),(-e_{i_2} + e_s),(-e_{i_1} + e_s)$, \;
 $T_2: \; (-e_{i_1} + e_t),(-e_{i_2} + e_t),(-e_{i_1} + e_t)$, }
 for some $e_s,e_t$.
 With the help of \ref{cross}) and \ref{posvebindes}) 
 we calculate that the points of  intersection of the corresponding sides of $T_1$ and $T_2$ 
 are
 $c_1 = (-e_{i_1}-e_{i_2}+e_s+e_t)$, $c_2 = (-e_{i_2}-e_{i_3}+e_s+e_t)$, and
 $c_3 = (-e_{i_3}-e_{i_1}+e_s+e_t)$. 
 Evidently, the points $c_1,c_2,c_3$ lie on the line     
 $[-e_{i_1}-e_{i_2}-e_{i_3}+e_s+e_t]$, which proves the claim.
\end{proof}


\section{Examples}
\subsection{Multi-Fano configuration}


Let us consider the Fano configuration ${\goth F}$ as a configuration given by
a quasi difference set 
${\goth F}=\DifSpace(C_7,\{0,1,3\})$  (cf. \cite{mald, Lipski}). 
Now we present {\em multi-Fano} configuration  
${\goth F^+}=\DifSpace(C_k \oplus C_7,\{ (0,0),(0,1),(0,3),(1,0) \})
\cong \DifSpace(C_k,{\{ 0,1 \}})\oplus {\goth F}$ 
as a particular case of multiplied Fano configuration. The configuration 
${\goth F^+}$
for $k=3$ is presented in Figure \ref{fig:multfano}. One can observe that 
this structure contains Pappus configuration.
We refer to a $\{ 2 \}$-part $\{ i \} \times C_7 = {\goth F}^+ \ciach{\{ 1 \}}(i)$ 
of ${\goth F}^+$ as to a Fano part of it
(cf. \ref{prop:partsofprod}).
\begin{figure}[!h]
    \begin{center}
    \includegraphics[scale=0.6]{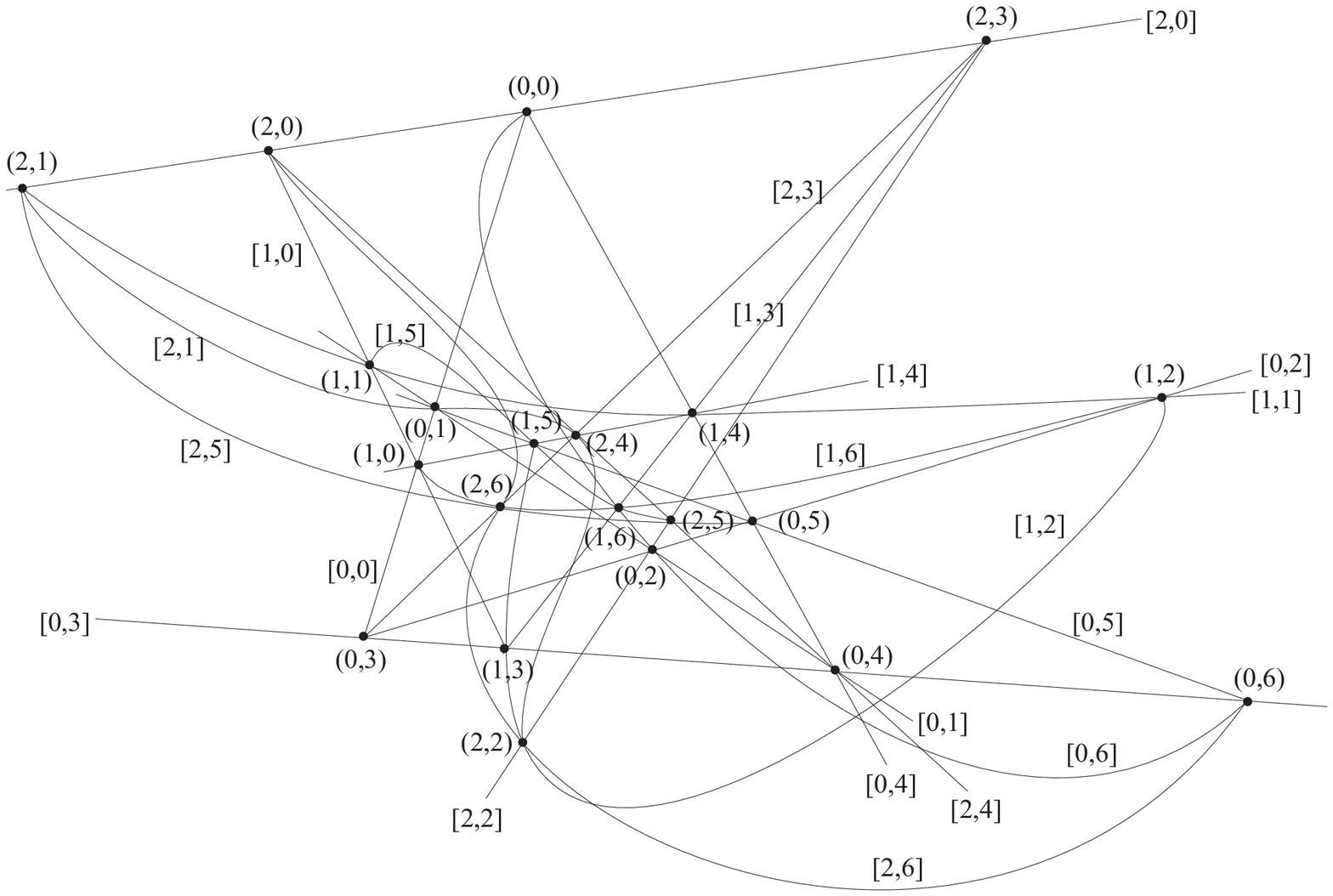}
    \end{center}
\caption{The multi-Fano configuration 
$\DifSpace(C_3 \oplus C_7,\{ (0,0),(0,1),(0,3),(1,0) \})$}
\label{fig:multfano}
\end{figure} 
\begin{lem}\label{lem:fanotransforms}
  If $f \in \Aut(\goth F^+)$, and $x,y$ are two points of $\goth M$, such that
  $f(x)=y$, then $f$ transforms the Fano's part of the neighborhood of $x$ into the 
  Fano's part of the neighborhood of $y$.
\end{lem}

\begin{proof}
  Let us assume $k>3$. Then the structure determined by points which are collinear 
  with $x = (i,a)$ is shown in Figure \ref{fig:multneighb}. 
  Let us observe, that the points of rank $4$ in the $(i,a)$ neighborhood form
  the Fano part ${\goth F}^+ \ciach{\{ 1 \}}(i)$, which yields our claim.

  If $k=3$, then three new lines appear: $[i+1,a],[i+1,a+1],[i+1,a+3]$ 
  in the neighborhood of $x$, and 
  the point $(i+1,a+6)$ is the only one of the rank $3$. There are three rank $2$ 
  lines and one rank $4$ line through the point $(i+1,a)$; and one rank $3$, 
  two rank $2$, one rank $4$ line through the point $(i+1,a+4)$. Thus the triangle 
  $\{ (i+1,a+d)\colon d=0,4,6 \}$ is uniquely determined by the geometry, 
  and as a consequence the lines 
  $\{[i+1,a+d]\colon d=0,1,3 \}$ are determined as well. Now the claim is evident. 
\end{proof}
\begin{figure}[!h]
    \begin{center}
    \includegraphics[scale=0.6]{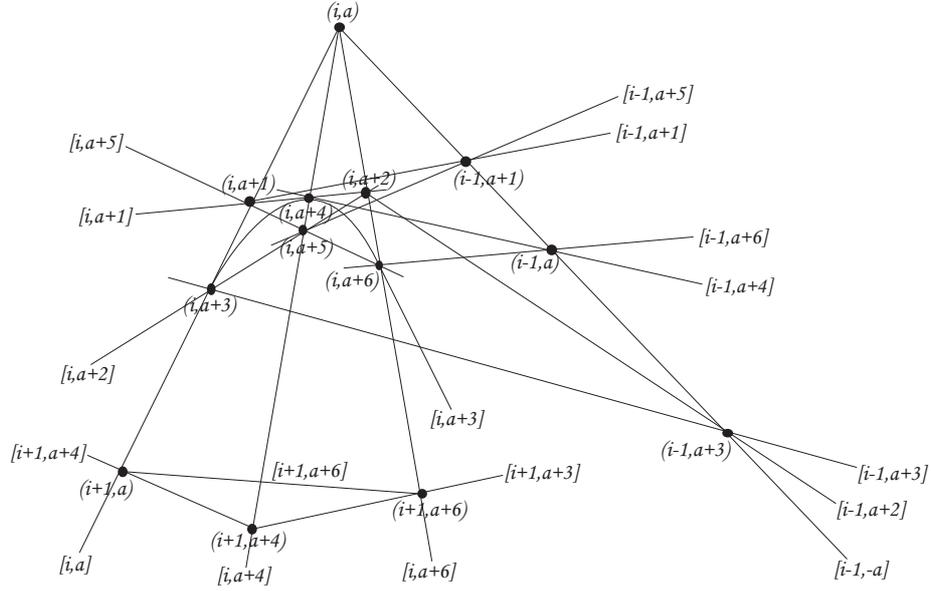}
    \end{center}
\caption{Neighborhoods of a point $(i,a)$ in multi-Fano configuration}
\label{fig:multneighb}
\end{figure} 

\begin{lem}\label{lem:fanopreserves}
  If $f \in \Aut(\goth F^+)$  and $f$ preserves 
  $\{0\}\times C_7$ then $f$ preserves 
  $\{i\}\times C_7$ for every $i \in \{0,1,\ldots,k-1\}$.
\end{lem}
\begin{proof}
  Let $f \in \Aut(\goth F^+)$; assume that $f$ preserves 
  $\{0\}\times C_7 = {\goth F}^+ \ciach{\{ 1 \}}(0)$. 
  Then every line $[0,l]$ is transformed by $f$ onto a line of the form
  $[0,l']$. Thus the image of the point $(1,a)\inc [0,l]$  is $(1,a')\inc [0,l']$
  for some $a',l'$. 
  Consequently, $f$ maps the substructure ${\goth F}^+ \ciach{\{ 1 \}}(i)$  
  of points of the level $i = 1$ onto itself. 
  Inductively, we get the same result for every $i = 0,1,\ldots,k-1$.
\end{proof}

\begin{lem}\label{lem:fanoautform}
  If $f \in \Aut({\goth F})$ then $f= \tau \circ g$, where $\tau$ is a translation of
  the group $C_7$, 
  and $g \in \otocz({\Aut({\goth F})},a)$. All the
  elements of $\otocz({\Aut({\goth F})},a)$
  are written in the table below.
\end{lem}

Collineations of the Fano configuration $\DifSpace(C_7,\{0,1,3\})$
which fix the point $(a)$

\noindent
\begin{tiny}
\begin{tabular}{cccccc|ccccccc}
$a+1$&$a+2$&$a+3$&$a+4$&$a+5$&$a+6$&$[a  ]$&$[a+1]$&$[a+2]$&$[a+3]$&$[a+4]$&$[a+5]$&$[a+6]$\\
\hline
$a+1$&$a+2$&$a+3$&$a+4$&$a+5$&$a+6$&$[a  ]$&$[a+1]$&$[a+2]$&$[a+3]$&$[a+4]$&$[a+5]$&$[a+6]$\\
$a+1$&$a+4$&$a+3$&$a+2$&$a+6$&$a+5$&$[a  ]$&$[a+1]$&$[a+3]$&$[a+2]$&$[a+6]$&$[a+5]$&$[a+4]$\\
$a+1$&$a+5$&$a+3$&$a+6$&$a+2$&$a+4$&$[a  ]$&$[a+5]$&$[a+2]$&$[a+3]$&$[a+6]$&$[a+1]$&$[a+4]$\\
$a+1$&$a+6$&$a+3$&$a+5$&$a+4$&$a+2$&$[a  ]$&$[a+5]$&$[a+3]$&$[a+2]$&$[a+4]$&$[a+1]$&$[a+6]$\\
$a+2$&$a+1$&$a+6$&$a+4$&$a+5$&$a+3$&$[a+6]$&$[a+1]$&$[a+5]$&$[a+3]$&$[a+4]$&$[a+2]$&$[a  ]$\\
$a+6$&$a+3$&$a+2$&$a+4$&$a+5$&$a+1$&$[a+6]$&$[a+3]$&$[a+2]$&$[a+1]$&$[a+4]$&$[a+5]$&$[a  ]$\\
$a+3$&$a+6$&$a+1$&$a+4$&$a+5$&$a+2$&$[a  ]$&$[a+3]$&$[a+5]$&$[a+1]$&$[a+4]$&$[a+2]$&$[a+6]$\\
$a+4$&$a+2$&$a+5$&$a+1$&$a+3$&$a+6$&$[a+4]$&$[a+1]$&$[a+2]$&$[a+5]$&$[a  ]$&$[a+3]$&$[a+6]$\\
$a+5$&$a+2$&$a+4$&$a+3$&$a+1$&$a+6$&$[a+4]$&$[a+2]$&$[a+1]$&$[a+3]$&$[a  ]$&$[a+5]$&$[a+6]$\\
$a+3$&$a+2$&$a+1$&$a+5$&$a+4$&$a+6$&$[a  ]$&$[a+2]$&$[a+1]$&$[a+5]$&$[a+4]$&$[a+3]$&$[a+6]$\\
$a+2$&$a+4$&$a+6$&$a+1$&$a+3$&$a+5$&$[a+6]$&$[a+1]$&$[a+3]$&$[a+5]$&$[a  ]$&$[a+2]$&$[a+4]$\\
$a+6$&$a+5$&$a+2$&$a+1$&$a+3$&$a+4$&$[a+6]$&$[a+5]$&$[a+2]$&$[a+1]$&$[a  ]$&$[a+3]$&$[a+4]$\\
$a+6$&$a+4$&$a+2$&$a+3$&$a+1$&$a+5$&$[a+6]$&$[a+3]$&$[a+1]$&$[a+2]$&$[a  ]$&$[a+5]$&$[a+4]$\\
$a+2$&$a+5$&$a+6$&$a+3$&$a+1$&$a+4$&$[a+6]$&$[a+2]$&$[a+5]$&$[a+3]$&$[a  ]$&$[a+1]$&$[a+4]$\\
$a+4$&$a+1$&$a+5$&$a+2$&$a+6$&$a+3$&$[a+4]$&$[a+1]$&$[a+5]$&$[a+2]$&$[a+6]$&$[a+3]$&$[a  ]$\\
$a+4$&$a+3$&$a+5$&$a+6$&$a+2$&$a+1$&$[a+4]$&$[a+3]$&$[a+2]$&$[a+5]$&$[a+6]$&$[a+1]$&$[a  ]$\\
$a+5$&$a+3$&$a+4$&$a+2$&$a+6$&$a+1$&$[a+4]$&$[a+2]$&$[a+3]$&$[a+1]$&$[a+6]$&$[a+5]$&$[a  ]$\\
$a+5$&$a+1$&$a+4$&$a+6$&$a+2$&$a+3$&$[a+4]$&$[a+5]$&$[a+1]$&$[a+3]$&$[a+6]$&$[a+2]$&$[a  ]$\\
$a+3$&$a+5$&$a+1$&$a+2$&$a+6$&$a+4$&$[a]$&$[a+2]$&$[a+5]$&$[a+1]$&$[a+6]$&$[a+3]$&$[a+4]$ \\
$a+3$&$a+4$&$a+1$&$a+6$&$a+2$&$a+5$&$[a]$&$[a+3]$&$[a+1]$&$[a+5]$&$[a+6]$&$[a+2]$&$[a+4]$\\
$a+2$&$a+3$&$a+6$&$a+5$&$a+4$&$a+1$&$[a+6]$&$[a+2]$&$[a+3]$&$[a+5]$&$[a+4]$&$[a+1]$&$[a]$\\
$a+6$&$a+1$&$a+2$&$a+5$&$a+6$&$a+3$&$[a+6]$&$[a+5]$&$[a+1]$&$[a+2]$&$[a+4]$&$[a+3]$&$[a]$\\
$a+4$&$a+6$&$a+5$&$a+3$&$a+1$&$a+2$&$[a+4]$&$[a+3]$&$[a+5]$&$[a+2]$&$[a]$&$[a+1]$&$[a+6]$\\
$a+5$&$a+6$&$a+4$&$a+1$&$a+3$&$a+2$&$[a+4]$&$[a+5]$&$[a+3]$&$[a+1]$&$[a]$&$[a+2]$&$[a+6]$
\end{tabular}
\end{tiny}

\begin{proof}
  It suffices to recall that the translations are automorphisms
  of every structure, which is defined by a quasi difference set.
\end{proof}

\begin{prop}\label{pr:multautomorf}
  Let ${\goth F^+}=\DifSpace(C_k \oplus C_7,\{(0,0),(0,1),(0,3),(1,0)\})$.
  \begin{sentences}
  \item
   If $7 \nmid k$, then the group $\Aut(\goth F^+)$ is isomorphic to $C_k \oplus C_7$.
  \item
    If $7 \mid k$, then the group $\Aut(\goth F^+)$ is isomorphic to 
    $C_3 \ltimes (C_k \oplus C_7)$.
  \end{sentences}
\end{prop}
\begin{proof}
  Generally, $\text{Tr}(C_k \oplus C_7) \subseteq \Aut(\goth F^+)$. 
  Let us take $g= \tau_{-f(0,0)} \circ f$, where $f \in \Aut(\goth F^+)$. 
  Then $g((0,0))=(0,0)$. 
  From \ref{lem:fanotransforms}, 
  every collineation $g \in \otocz({\Aut({\goth F^+})},{(0,0)})$ 
  preserves the Fano substructure
  in the neighborhood of the point $(0,0)$. From \ref{lem:fanopreserves},
  the Fano substructure is preserved on every of $i$ levels, where $i= 0,1,\ldots,k-1 $. 
  In view of \ref{pr:extautinmult}, there is a map $\varphi = (\varphi',\varphi'')$, 
  such that 
    $\varphi'(i,a)=(i,f'_i(a))$ and $\varphi''([i,b])=[i,f''_i(b)]$, 
  where $(f_i',f_i'')=f_i \in \Aut({\goth F})$, $i\in C_k$.
  From  \ref{pr:extautinmult}\eqref{extautinmult:cas2}, $\varphi$ is a 
  collineation of $\goth F^+$ iff $f'_{i+1}=f''_i$. Analyzing the table in 
  \ref{lem:fanoautform} we find that there are two different collineations,
  which have a chance to satisfy the condition above;
  these are: 
  $\varphi'_1(x)=2x$, $\varphi'_2(x)=4x$. 
  It is easy to compute that
    $\varphi''_1 = \tau_6\varphi'_1$, 
    $\varphi''_2 = \tau_4\varphi'_2$. 
  Let $f'_0 = \varphi'_1$, then $f''_0 = \tau_6\varphi'_1$. By induction we get
  that $f'_i = \tau_{6i}\varphi'_1$, and then $f''_i = \tau_{6(i+1)}\varphi'_1$.
  In particular,
    $f''_{k-1} = \tau_{6k}\varphi'_1 = f'_0 = \varphi'_1$. 
  Therefore $\tau_{6k} = \id$ i.e. 
  $6k\equiv 0$ holds in $C_7$ and thus $7 \mid k$. 
  We will get the same result if we consider $f'_0 = \varphi'_2$. 
  To close the proof we observe that 
    $G = \{ \varphi'_1,\varphi'_2,\id \} \cong C_3$ 
  and 
    $f\tau_{(j,b)}f^{-1}(i,a) = \tau_{f(j,b)}(i,a)$ for $f \in G$.
\end{proof}


\subsection{Multi-Pappus configuration}

Since $\DifSpace((C_3)^2, {\cal D}_2)$ is simply the Pappus configuration,
a structure of the form $\DifSpace((C_3)^n, {\cal D}_n)$
will be called a {\em multi-Pappus} configuration.
Figure \ref{fig:multpap0} illustrates the structure
$\DifSpace((C_3)^3, {\cal D}_3)$.
Note that \ref{cor:semidirprod} cannot be used to characterize the group 
of automorphisms of a multiplied Pappus configuration.
Below, we shall describe the automorphism group of the structure 
  $\DifSpace((C_3)^n, {\cal D}_n)$, 
  where ${\cal D}_n$ is the canonical quasi difference set in the group 
  $(C_3)^n$.


\begin{lem}\label{lem:multpaprigid}
  Let ${\goth M} = \DifSpace((C_3)^n, {\cal D}_n)$,
  $f = (f',f'')\in\Aut({\goth M})$.
  If 
  \begin{enumerate}[\rm(a)]
  \item\label{multpaprigid:cond}
    $f''$ fixes every line through $o$ and $f'$ fixes every point on $L$
  \end{enumerate}
  for some  point $o$ of $\goth M$ and some line $L$ through $o$,
  then $f$ is the identity automorphism.
\end{lem}
\begin{proof}
  Let $M$ be any line through $o$, $L\neq M$, and $o\neq x\inc M$. 
  From \eqref{multpaprigid:cond} there is the unique point $y\inc L$ with $o\neq y\collin x$.
  From assumptions, $f''(M)=M$ and $f'(y)=y$, which gives $f'(x)=x$.
  Thus we proved that $f'$ fixes every point collinear with $o$.

  By the duality principle ($\goth M$ is self-dual!), 
  $f''$ fixes every line which crosses $L$.
  Therefore $f$ satisfies \eqref{multpaprigid:cond} for every
  point $x$ collinear with $o$ and every line through $x$.
  From the connectedness of $\goth M$ we get our statement.
\end{proof}

\begin{prop}\label{prop:multpapaut}
  Let ${\goth M} = \DifSpace((C_3)^n,{\cal D}_n)$ with $n > 2$
  and let ${\goth G} = \Aut({\goth M})$.
  Then ${\goth G}\cong S_{n+1}\ltimes (C_3)^n$,
    where the action of $S_{n+1}$ on $(C_3)^n$ is defined in 
    \ref{cor:semidirprod}.
\end{prop}
\begin{proof}
  Let $f\in{\goth G}$.
  Clearly, $\goth M$ has a point transitive group of automorphisms: translations.
  Therefore there is $q \in (C_3)^n$ such that $f = \tau_q\circ f_0$
  and $f_0 \in \otocz({\goth G},\theta)$.
  Then $f_0 = (f'_0,f''_0)$ determines a permutation $\alpha\in S_{n+1}$ such that 
  $f''_0$ maps the line $[-e_i]$ onto $[-e_{\alpha(i)}]$ for 
  $i=0,\ldots,n$.
  From \ref{lem:permtocollin} there is an automorphism 
  $f_{\alpha^{-1}}\in \otocz({\goth G},\theta)$
  associated with $\alpha^{-1}$; we set $g = f_{\alpha^{-1}} \circ f_0$
  and then $g\in \otocz({\goth G},\theta)$ 
  preserves every line through $\theta$.
  In particular, $g = (g',g'')$ permutes points on $[e_0]$ so,
  it determines a permutation $\beta\in S_n$ such that
  $g'(e_i) = (e_{\beta(i)})$ for $i = 1,\ldots,n$.

  If $\beta = \id$, then from \ref{lem:multpaprigid} we infer that $g$ is
  the identity automorphism.
  Assume that $\beta \neq \id$ so, $i \neq j = \beta(i)$ for some $i,j$.
  and thus $g'(e_i) = e_j\neq e_i$.
  Considering pairs of lines $[e_0],[-e_i]$ and $[e_0],[-e_j]$ 
  from \eqref{wz:multpapneb} we infer that 
  $g'(-e_i) = (-e_i + e_j)$ and 
  $g'(-e_j + e_i) = (-e_j)$.
  Again from \eqref{wz:multpapneb} considering $[-e_i],[-e_j]$ we get 
  ${g'}^{-1}(-e_j + e_i) = (-e_j)$, which gives, finally, $g'(e_j)=e_i$.
  Thus we proved: $\beta$ interchanges $i$ with $j$.

  Let $k\neq i,j,0$ and $k \leq n$.  
  Consider the lines $[-e_k]$, $[e_0]$, and $[-e_i]$.
  From \eqref{wz:multpapneb} we get 
  $g'(-e_k + e_i) = -e_k + e_j = {g'}^{-1}(-e_k + e_i)$
  and 
  $g'(-e_i) = (-e_i + e_j)$.
  Note that $(-e_i + e_j) \collin (-e_k + e_j)$ and $(-e_i) \collin (-e_k)$
  and thus $g'(-e_k) = (-e_k + e_j)$ which yields
  a contradiction: $(-e_k) = (-e_k + e_i)$.
  Finally, we come to $\beta = \id$, which closes the proof.
\end{proof} 


\subsection{Splitting of the multi-Pappus configuration and a 
cyclic projective plane}

Let us consider the structure
\begin{equation}\label{str:papplusproj}
  {\goth M} = \DifSpace((C_3)^k,{\cal D}_k)\oplus \DifSpace(C_n,D),
\end{equation}
where $\DifSpace((C_3)^k,{\cal D}_k)$ is a configuration given in 
\ref{prop:multpapaut} and ${\goth N} = \DifSpace(C_n,D)$ is the 
cyclic projective plane $PG(2,q)$, given by a Singer difference set 
  $D = \{ d_0,d_1,\ldots,d_q \}$ 
in the group 
  $C_n$, $n=q^2+q+1$
(cf. \cite{sing, Lipski}).  
We usually adopt $d_0=0$.
According to the general theory, the projective cyclic plane $PG(2,q)$ may be
determined by $q+1$ distinct difference sets $D^1,D^2,\ldots ,D^{q+1}$.
One can notice, that configurations 
  ${\goth M}^i = \DifSpace((C_3)^k,{\cal D}_k)\oplus \DifSpace(C_n,D^i), 
  i=1,2,\ldots ,q+1$ 
are not necessarily isomorphic. Indeed, for instance:
  $$\DifSpace(C_3^2,{\cal D}_2)\oplus \DifSpace(C_{13},\{0,1,3,9\})\ncong
     \DifSpace(C_3^2,{\cal D}_2)\oplus \DifSpace(C_{13},\{0,2,8,12\}).$$
The structure $\goth M$ can be also represented in the form:
  $$\underbrace{\DifSpace(C_3,{\cal D}_1) \oplus (\ldots \oplus
    (\DifSpace(C_3,{\cal D}_1)}_{k \; \mathrm{times}} 
     \oplus \DifSpace(C_n,D))\ldots )$$ 
(cf. \ref{pr:associatprod}); therefore, 
we can use several times \ref{lem:papinscneighb} to get an exact description 
of the neighborhood of a point of $\goth M$.
In particular,  we can present every point $p \in \goth M$ in the form 
  $$p=(x_k,\ldots,x_1,y),$$ 
where $x_k,\ldots,x_1\in C_3$, $y\in C_n$.  
In this notation, the points and the lines of 
$\otocz({\goth M},\theta)$ -- the neighborhood of 
$\theta = (0,\ldots,0) \in (C_3)^k\times C_n$, are the following:
\begin{eqnarray*}
   q_{i,j}=(0,\ldots,0,-d_i+d_j), & 
  \\ 
   p'_{m,i}=(0,\ldots,1_m,0,\ldots,-d_i), & &
   p''_{m,i}=(0,\ldots,2_m,0,\ldots,d_i)  
  \\
   l_i = [0,\ldots,0,-d_i], & &
   l''_m = [0,\ldots,2_m,0,\ldots,0]
  \\      
   \textrm{for}\; d_i,d_j\in D; \; i,j=0,\ldots,q; \; i\neq, j; & &
   m=1,\ldots,k   
  \\
   p'_{s,q+r}=(0,\ldots,1_s,0,\ldots,2_r,0,\ldots,0), & &
   p''_{s,q+r}=(0,\ldots,2_s,0,\ldots,1_r,0,\ldots,0)    
  \\
   \textrm{for} \; 
   s=2,\ldots,k;\; r=1,\ldots,s-q+2 .
\end{eqnarray*} 
\begin{lem}\label{lem:prespapandpg}
  Let ${\goth M}$ be the structure defined in \eqref{str:papplusproj}, and
  $F\in \otocz({\Aut(\goth M)},\theta)$.  
  Then, $F$ leaves invariant the multi-Pappus subconfiguration 
  (a $\{ 1,\ldots,k \}$-part, or multi-Pappian part of $\goth M$)
  \\[1.5ex]
  \centerline{
    $(C_3)^k \times \{0\} = {\goth M} \ciach{\{ k+1 \}} (0)
    \cong \DifSpace((C_3)^k,{\cal D}_k)$,}
  and $F$ leaves invariant the cyclic projective subplane
  (a $k+1$-part, or "projective" part of $\goth M$)
  \\[1.5ex]
  \centerline{
    $\{(0,\ldots,0)\}\times C_n = 
    {\goth M}  \ciach{\{ 1,\ldots,k \}}  \underbrace{(0,\ldots,0)}_{k \; \mathrm{times}}
    \cong \DifSpace(C_n,D)$} 
  (cf. \ref{prop:partsofprod}) as well.
  \par
  Moreover, $F$ determines two permutations: $\alpha \in S_q$ and $\beta \in S_k$
  such that 
  $F(l''_m)=l''_{\beta(m)}$ and then $F(p'_{m,0})=p'_{\beta(m),0}$,
  and $F(l_i) = l_{\alpha(i)}$ for $m=1,\ldots,k$, $i=1,\ldots,q$.
\end{lem}
\begin{proof} 
  Let us take a closer look at the neighborhood $\otocz({\goth M},\theta)$.
  Recall, that every line and point of $PG(2,q)$ is of the rank $q+1$. 
  Then, the maximal rank of the point or line in $\otocz({\goth M},\theta)$
  equals $q+k+1$. Assume $q\ge 3$.
  The only part of $\otocz({\goth M},\theta)$, 
  where not passing through $\theta$ lines of the rank greater than $3$ appear, 
  is precisely the projective cyclic subplane 
    $S=\{(0,\ldots,0)\}\times \DifSpace(C_n,D)$.
  If $q=2$ then only this subplane consists of not passing 
  through $\theta$ lines, which are incident with three points of the rank $3+k$. 
  Therefore, it must be invariant under $F$ so, in particular, the family of lines 
    $l_i=[0,\ldots,-d_i]$ for $i=0,\ldots,q,\; d_i\in D$ 
  is preserved.
  
  With the help of \ref{lem:papinscneighb} we see that 
  $p'_{m,0}\inc l'_0$ are the only points of rank $q+k+1$ 
  which simultaneusly: lie on one of the lines 
  $l_i$, do not belong to some rank $3$ line, and are not from $S$. 
  Thus $F$ must preserve these points and, consequently, the line $l_0$. 
  Therefore, there exists $\alpha \in S_q$ such that $F(l_i) = l_{\alpha(i)}$
  for $i > 0$.
  Furthermore, $F$ leaves the set
  $\{l''_m=[0,\ldots,2_m,0,\ldots,0]\colon m=1,\ldots,k\}$  
  invariant  so, there exists
  $\beta \in S_k$ such that $F(l''_m)=l''_{\beta(m)}$. 
  
  The points of $\otocz({\goth M},\theta)$ which belong to 
    $\DifSpace((C_3)^k,{\cal D}_k)\times \{0\}$
  are the following: 
    $p'_{m,0}$, $p''_{m,0}$, $p'_{s,q+r}$, and $p''_{s,q+r}$, where 
    $m=1,\ldots,k$; $s=2,\ldots,k$; $r=1,\ldots,s-q+2 $. 
  Note, that $p''_{m,0}, p'_{s,q+r}, p''_{s,q+r}$ are the only points on 
  $l''_m$, which are connected with one of $p'_{m,0}$. 
  Besides, every line of the rank $2$ passing through $p''_{m,0}\inc l''_m$
  consists of $q+k+1$ rank points, but some lines passing through 
  $p'_{s,q+r}$ or through $p''_{s,q+r}$ contain elements of the rank less 
  than $q+k+1$. It means that the points $p''_{m,0}$ and then
  $p'_{m,0}$ are permuted under the same permutation $\beta \in S_k$ which
  acts on the set of the lines $l''_m$.   
\end{proof}
\begin{lem}\label{lem:syminpapplane}
  Let $\beta\in S_k$. We define the map $G_{\beta}$ on $C_3^k\oplus C_n$ 
  by the formula
  $G_{\beta}((x_k,\ldots,x_1,y)) = (x_{\beta(k)},\ldots,x_{\beta(1)},y)$.
  Then $G_{\beta}\in \Aut(\goth M)$, $G_{\beta}(\theta) = \theta$, 
  and $G_{\beta}(p'_{m,0})=p'_{\beta(m),0}$. 
\end{lem}
\begin{proof}
  We can write 
    $G_{\beta} = G_1\oplus {\id}_{C_n}$, 
  where 
    $G_1 = (f',f'')$ with $f' = f'' = h$ and 
    $h((x_k,\ldots,x_1)) = (x_{\beta(k)},\ldots,x_{\beta(1)})$. 
  From \ref{lem:syminsegre}, $G_1\in \Aut(\DifSpace(C_3^k,{\cal D}_k))$, and 
  from \ref{pr:autinsegre},
    $G_{\beta}\in \Aut({\goth M})$.
\end{proof}
\begin{lem}\label{lem:symetrofpapfan}
  Let $G_{\beta}$ be the map defined in \ref{lem:syminpapplane}, let  
  ${\cal G}_0 = \{G_{\beta}\colon \beta\in S_k\}$,
  and let 
  ${\cal G} = \{\tau_a \circ g \colon g\in {\cal G}_0, a \in (C_3)^k \times C_n\}$.
  Then ${\cal G}_0 \cong S_k$ and $\cal G$ is the group isomorphic 
  to the semidirect product $S_k\ltimes (C_3^k\oplus C_n)$.
\end{lem}
\begin{proof}
  It suffices to note that $G_{\beta_1}\circ G_{\beta_2}=G_{\beta_1 \beta_2}$ 
  and 
  $G_{\beta}\circ \tau_a\circ G_{\beta}^{-1}(u)=
   \tau_{G_{\beta}(a)}(u)$, 
  where
    $a,u\in C_3^k\oplus C_n$.
\end{proof}
\begin{lem}\label{lem:permpointprim}
  Under assumptions of \ref{lem:prespapandpg} if, additionally, $\beta = \id$  
  (i.e. the permutation $\beta$ defined in \ref{lem:prespapandpg} is the identity), 
  then every point $p''_{m,0}, p'_{s,q+r}, p''_{s,q+r}$ is fixed by $F$,
  for all $m=1,\ldots,k$; $s=2,\ldots,k$; $r=1,\ldots,k-1$. 
  Moreover, the permutation $\alpha$ of $\{1,\ldots,q\}$ 
  satisfies: $F(p'_{m,i})= p'_{m,\alpha(i)}$  
  for $i=1,\ldots,q$.
\end{lem}
\begin{proof}
  Let $m\in \{1,\ldots,k\}$. Note, that the point 
  $p''_{m,0}$ is the only one on the line $l''_m$, which is collinear with
  $p'_{m,0}$. Now, let
  $s\in \{2,\ldots,k\}$, $r\in \{1,\ldots,s-q+2\}$.
  Then,  
  $p'_{s,q+r}\inc l''_r$ and $p''_{s,q+r}\inc l''_s$ can be described
  as the (unique) points connected with $p'_{s,0}$ and $p'_{r,0}$ respectively.
  Thus, with the help of \ref{lem:prespapandpg}, we obtain
  \\[1.5ex]
  \centerline{
   $p'_{m,0}= F(p'_{m,0})\collin F(p''_{m,0})\inc F(l''_m)=l''_m$,} 
  \\[1.2ex]
  \centerline{
   $p'_{s,0}= F(p'_{s,0})\collin F(p'_{s,q+r})\inc F(l''_r)=l''_r$,}
  \\[1.2ex]
  \centerline{
   $p'_{r,0}= F(p'_{r,0})\collin F(p''_{s,q+r})\inc F(l''_s)=l''_s$.}
  \vskip1.5ex\noindent
  Hence,
  $F(p''_{m,0})=p''_{m,0}$,  $F(p'_{s,q+r})=p'_{s,q+r}$ and 
  $F(p''_{s,q+r})=p''_{s,q+r}$.
  
  The point $p'_{m,i}$ can be characterized as the unique 
  point connected with $p'_{m,0}$, which lies on $l_i$. Consequently,
  \\[1.5ex]
  \centerline{
  $p'_{m,0}= F(p'_{m,0})\collin F(p'_{m,i})\inc F(l_i)=l_{\alpha(i)}$,}
  \vskip1.5ex\noindent
  and then $F(p'_{m,i})= p'_{m,\alpha(i)}$.
%
\end{proof}
\begin{lem}\label{lem:permpintbis}
  Under assumptions of \ref{lem:permpointprim}, and the condition
  \begin{enumerate}[\rm(a)]
  \item\label{triangle}
    for every, except at most one, $d_i\in D$ there exist $d_j,d_r\in D$
    such that 
    $$d_i + d_j + d_r = 0$$ 
  \end{enumerate}
  the permutation $\alpha$ given in \ref{lem:prespapandpg} satisfies the following: 
    $F(p''_{m,i})= p''_{m,\alpha(i)}$, 
    $F(q_{i,j})= q_{\alpha(i), \alpha(j)}$, 
  for $m=1,\ldots,k$; $i,j=1,\ldots,q$, $i\neq j$. 
\par\noindent
  Consequently, if $\alpha = \id$, 
  then $F$ is the identity on $\otocz({\goth M},\theta)$.
\end{lem}
\begin{proof}
  Consider arbitrary  
  $m\in \{1,\ldots,k\}$.
  Let us draw attention to the lines, which join the points from the set  
  $\{ p''_{m,i}\colon i=1,\ldots,q \} =: P''$ 
  with the points  
  $\{ p'_{m,i}\colon i=1,\ldots,q \} =: P'$.
  The condition \eqref{triangle} yields that there is at most one 
  pair $(x,y) \in P'' \times P'$ such that $x$ is unconnected with the points
  from $P'$ and $y$ is unconnected with the points from $P''$.
  What is more, every remaining 
  point $p''_{m,i}$ from $P''$ is collinear with two points from $P'$:
  $p'_{m,j}$ and $p'_{m,r}$, where $d_i + d_j + d_r = 0$. 
  Note that, if $d_i + d_{j_t} + d_{r_t} = 0$ for $t = 1,2$ then 
  $\{d_{j_1},d_{r_1} \} = \{ d_{j_2}, d_{r_2} \}$
  and thus the above points $p'_{m,j}$ and $p'_{m,r}$ are uniquely determined by $p''_i$.
  Therefore, 
  \\[1.5ex]
  \centerline{
  $l''_m = F(l''_m)\inc F(p''_{m,i})\collin F(p'_{m,j}) = p'_{m,\alpha(j)},$}
   \\[1.2ex]
  \centerline{
  $l''_m = F(l''_m)\inc F(p''_{m,i})\collin F(p'_{m,r}) = p'_{m,\alpha(r)},$}
  \vskip1.5ex\noindent
  and then $F(p''_{m,i})= p''_{m,\alpha(i)}$.

 \par 
  From \ref{lem:papinscneighb} we get, that $q_{i,j}\inc l_i$ and
  $q_{i,j}\collin p''_{m,j}$. Consequently, 
  \\[1.5ex]
  \centerline{
  $F(q_{i,j})\inc F(l_i)=l_{\alpha(i)}$ and
  $F(q_{i,j})\collin F(p''_{m,j})=p''_{m,\alpha(j)}$;}
  \vskip1.5ex\noindent
  thus $F(q_{i,j}) = q_{\alpha(i), \alpha(j)}$.  
\end{proof}
%
%
Remind, that the symbol $\otocz({\goth M},q)$ with $q\in (C_3)^k \times C_n$
means the substructure of $\goth M$, 
which contains the points collinear with $q$ and the lines among them. 
If $F(q) = q$, then we write 
$\alf(F,q)$ for the permutation $\alpha$  of $\{ 1,\ldots,q \}$ 
and $\bet(F,q)$ for the permutation $\beta$ of $\{ 1,\ldots,k \}$
in accordance with \ref{lem:prespapandpg} determined by   
the permutation $F\restriction{\otocz({\goth M},q)}$ of the points of $\otocz({\goth M},q)$.
\begin{lem}\label{lem:preslinethru}
  Let the condition \eqref{triangle} of \ref{lem:permpintbis} be satisfied for
  $\goth M$ defined in \eqref{str:papplusproj}.
  If $F\in \Aut(\goth M)$ preserves every line passing through $q$ 
  (in particular, $F(q) = q$),
  then the permutations $\bet(F,s)$ and $\alf(F,q)$ determined by $F$, defined in
  \ref{lem:permpointprim} and \ref{lem:permpintbis}, are identities,
  and thus $F$ is the identity on $\otocz({\goth M},q)$.
\end{lem}
\begin{proof}
  Clearly, the translation $g = \tau_q$ maps $\theta$ to $q$
  and thus it maps $\otocz({\goth M},\theta)$ onto $\otocz({\goth M},q)$.
  We set $G = g^{-1} \circ F \circ g$. 
  Then   $G \in \Aut({\goth M})$ and $G(\theta) = \theta$. 
  From assumption,  $G$ maps every line through $\theta$ onto itself.
  In particular, the lines $l''_m$ with $m=1,\ldots,k$ are preserved under $G$,
  so $\bet(G,\theta) = \id$, $\alf(G,\theta) = \id$, and thus, by \ref{lem:permpintbis}, 
  $G$ is the identity on $\otocz({\goth M},\theta)$.
  Therefore, $F = g \circ G \circ g^{-1}$ is the identity on $\otocz({\goth M},q)$.
\end{proof}
\begin{lem}\label{lem:presallneighb}
  Let the condition \eqref{triangle} of \ref{lem:permpintbis} be satisfied for
  $\goth M$ defined in \eqref{str:papplusproj},
  and let $F\in \Aut({\goth M})$  fix all the points of $\otocz({\goth M},q)$.
  If $q'\in \otocz({\goth M},q)$, then $F$ fixes the points of $\otocz({\goth M},q')$ as well. 
\end{lem}
\begin{proof}
  Observe, that every point $q'\in \otocz({\goth M},q)$ has rank greater or equal
  than $q + k$ and the rank of each point in $\goth M$ equals $q + k + 1$.
  It means, that $q + k$ from $q+ k + 1$ lines passing through $q'$ are contained in
  $\otocz({\goth M},q)$. Since $F$ fixes $\otocz({\goth M},q)$, these $q + k$ lines 
  remain invariant under $F$. Then, the last line through $q'$ is also fixed. 
  Applying \ref{lem:preslinethru} we obtain that $F$ is the identity on 
  $\otocz({\goth M},{q'})$.
\end{proof}
Since $\goth M$ is connected, 
combining \ref{lem:permpintbis} and \ref{lem:presallneighb} we obtain
\begin{cor}\label{cor:presallneighb}
  Let the condition \eqref{triangle} of \ref{lem:permpintbis} be satisfied for
  $\goth M$ defined in \eqref{str:papplusproj}
  and let $F\in \Aut({\goth M})$ and $q$ be a point of $\goth M$.
  If $F(q) = q$ and $F$ preserves every line through $q$,
  then $F = \id$.
\end{cor}

Now, we determine the automorphisms group of $\goth M$ defined in
\eqref{str:papplusproj} 
with  ${\goth N} = {\goth F} = \DifSpace(C_7,\{0,1,3\})$ being the cyclic projective 
plane $PG(2,2)$.
The obtained structure may be considered as a composition of
multi-Pappus and Fano configurations.
Note, that other difference sets in $C_7$ give us structures
isomorphic to $\goth M$.
\begin{prop}\label{prop:autofpapusfano}
  Let 
    ${\goth M} = \DifSpace(C_3^k,{\cal D}_k)\oplus \DifSpace(C_7,\{ 0,1,3 \})$
  with $k \geq 2$.
  Then the group $\Aut({\goth M})$ is isomorphic to $S_k\ltimes (C_3^k\oplus C_7)$.
\end{prop}
\begin{proof}
  Let $F$ be an automorphism of $\goth M$ and $g=\tau_{-F(\theta))}\circ F$. Then
  $g(\theta) = \theta$ and $g\in \Aut({\goth M})$. In view of 
  \ref{lem:prespapandpg}, $g$ leaves $\{p'_{m,0}\colon m=1,\ldots,k\}$ invariant. 
  We set 
    $h = g \circ G^{-1}_{\beta}$,   
  where $G_{\beta}$ is the  map defined in \ref{lem:syminpapplane}
  and $\beta = \bet(g,\theta) \in S_k$.
  Now, $h\in \Aut({\goth M})$, $h(\theta) = \theta$ and 
  $\bet(h,\theta) = \id$.

  Consider the permutation $\alpha = \alf(h,\theta)$ of $\{1,2\}$ 
  determined by $h$,
  in accordance with \ref{lem:permpointprim} and \ref{lem:permpintbis}. 
  It follows from \ref{lem:permpointprim}, that $h$ permutes the points from 
  $\{p'_{m,i}\colon i=1,2\}$ for all $m=1,\ldots,k$.
  Note, that each of the points $p'_{m,1},p'_{m,2}$ has different rank, 
  and thus $\alpha$ is the identity on $\{1,2\}$.
  This together with \ref{lem:permpointprim} yields that $h$ preserves
  every line through $\theta$.
  From \ref{cor:presallneighb} we get $h = \id$ so, $g = G_{\beta}$.
  With \ref{lem:symetrofpapfan} we close the proof.
\end{proof}

Adopt ${\goth N} = PG(2,3)$ instead of Fano configuration. Then, we obtain
\\[1ex]
\centerline{
${\goth M}^i = \DifSpace(C_3^k,{\cal D}_k)\oplus \DifSpace(C_{13},D^i)$,}
\vskip1ex\noindent
where $i=1,2,3,4$ and 
  $D^1 = \{0,1,3,9\}$, $D^2 = \{0,2,8,12\}$, $D^3 = \{0,6,10,11\}$, $D^4 = \{0,4,5,7\}$. 
Note, that the condition \eqref{triangle} from \ref{lem:permpintbis} is satisfied
for all of the $D^i$, but in different manners: 
for every $d_i\in D^1$ there exist $d_j,d_r\in D^1$ such that $d_i+d_j+d_r = 0$, 
and in $D^2,D^3,D^4$ there is exactly one $d_s$ for which do not exist such $d_j,d_r$.
Thus, the configurations ${\goth M}^j$, $j=2,3,4$ are isomorphic to each 
other (for $\phi(x)=3x$ we get $\phi(D^2)=D^3,\phi(D^3)=D^4,\phi(D^4)=D^1$), 
but are not isomorphic to ${\goth M}^1$. Then, it suffices to consider 
${\goth M}^1$ and ${\goth M}^2$. 
\begin{prop}
  Let 
    ${\goth M}^i = \DifSpace(C_3^k,{\cal D}_k)\oplus \DifSpace(C_{13},D^i)$,
  where $i=1,2$ and $D^1 = \{0,1,3,9\}$, $D^2 = \{0,2,8,12\}$.
  Then, the group $\Aut({\goth M}^i)$ is isomorphic to
     $S_k\ltimes (C_3^k\oplus C_{13})$.
\end{prop}
\begin{proof}
  Let $F$ be an automorphism of ${\goth M}^i$ and $g = \tau_{-F(\theta)}\circ F$. 
  Then $g(\theta) = \theta$ and $g\in \Aut({\goth M}^i)$. 
  In view of \ref{lem:prespapandpg}, $g$ leaves  
  $\{p'_{m,0}\colon m=1,\ldots,k\}$ 
  invariant.  We set 
  $h=g\circ G^{-1}_{\beta}$, 
  where $G_{\beta}$ is the  map defined in \ref{lem:syminpapplane}
  and $\beta = \bet(g,\theta) \in S_k$.
  Then 
    $h\in \Aut({\goth M}^i)$, $h(\theta) = \theta$, and $\bet(h,\theta) = \id$;
  from \ref{lem:permpointprim}, $h$ preserves $\{l''_m\colon m=1,\ldots,k\}$.
  
  Let $\alpha = \alf(h,\theta)$ be the permutation determined by $h$ in 
  accordance with \ref{lem:permpointprim}; then from \ref{lem:permpintbis} we 
  find that 
    $h(q_{i,j}) = h_\alpha(q_{i,j}) := q_{\alpha(i),\alpha(j)}$ for all $i,j$.

  For  ${\goth M}^1$ note that if $\alpha\neq \id$, then
  $h_\alpha$ does not preserve the collinearity in 
    $S = \{(0,\ldots,0)\}\times \DifSpace(C_{13},D^1)$. 
  For example: $q_{1,2}$, $q_{0,2}$, $q_{3,1}$, $q_{2,1}$ form a line in $S$, 
  but 
    $q_{\alpha(1),\alpha(2)}$, $q_{0,\alpha(2)}$, $q_{\alpha(3),\alpha(1)}$, 
    $q_{\alpha(2),\alpha(1)}$ 
  do not, unless $\alpha = \id$.  

  In ${\goth M}^2$ for all $m=1,\ldots,k$ we have: 
  $p'_{m,1}$ are points of rank $5$  laying on a line of rank $3$, 
  $p'_{m,2}$ are points of rank $5$  not laying on any line of rank $3$, and
  $p'_{m,3}$ are points of rank $6$; so $h$ fixes these points
  and thus $\alpha = \id$.

  In both cases, $h$ fixes all the lines through $\theta$ so, 
  from \ref{cor:presallneighb} we get $h = \id$ and thus $g = G_{\beta}$. 
  Finally, applying \ref{lem:symetrofpapfan} we finish the proof.
\end{proof}


\subsection{A power of cyclic projective plane}

Let ${\goth P} = \DifSpace(C_k,{\cal D})$ be a cyclic projective plane
determined by a difference set $\cal D$ in the group $C_k$
($k = q^2+q+1$, where $q + 1$ is both the size of a line 
and the degree of a point of $\goth P$).

Now, let us draw our attention to the following structure
\begin{equation}\label{multisinger}
  {\goth M} := \underbrace{\goth P\oplus \goth P\oplus \ldots
       \oplus\goth P}_{n \; \mathrm{times}} =: {\goth P}^n
\end{equation}
Remind that 
${\goth M} = \DifSpace({(C_k)^n},D)$,
where $D = {\cal D}\uplus\ldots\uplus{\cal D}$.
Let us introduce a few, useful in further, definitions. Namely:
  \begin{eqnarray*}
 & & \supp(x) := \{ i \in\{ 1,\ldots,n \} \colon x_i \neq 0 \} \textrm{ for } x \in (C_k)^n,
 \\
 & & \Xi_{\alpha}^{\gamma}:= 
 \{x\in(C_k)^n\colon |\supp(x)|=\gamma \textrm{ and if } i\in \supp(x) 
 \textrm{ then } x_i=\alpha\}, 
 \\
 & & \Xi_{\{\alpha\},\{\beta\}}^{1}:= 
   \Xi^1_{\alpha} \cup \Xi^1_{\beta},
 \\
 & & \Xi_{\{\alpha,\beta\}}^{2}:= 
 \{x\in(C_k)^n\colon |\supp(x)|=2 ,\;
 \{ i,j \} = \supp(x) 
 \implies \{x_i,x_j\}= \{\alpha,\beta\}\},  
 \\
 & & P_i:= \{x\in(C_k)^n\colon \supp(x)= \{i\}, \textrm{ or } x=\theta\}.  
  \\
 & & S_i = \{(x)\colon x \in P_i \},\;\; \flines_i =\{ [x]\colon x \in P_i \}, \;\; i = 1,\ldots,n.
 \end{eqnarray*}
  For $i_1 \neq i_2$ we have 
    $\flines_{i_1} \cap \flines_{i_2} = \{ [\theta] \}$
  and
    ${S}_{i_1} \cap {S}_{i_2} = \{ (\theta) \}$.
  The sets $S_i$ and $\flines_i$  
  consist of points and lines respectively, which form a projective plane embedded
  in $\otocz({\goth P}^n,{(\theta)})$. There are $n$ such 
  planes with  the common line $[\theta]$, and the common point $(\theta)$ in
  $\otocz({\goth P}^n,{(\theta)})$. 
  Note that, the degree of the point $(\theta)$ in 
    $\otocz({{\goth P}^n},{(\theta)})$
  equals $q n + 1$, and it is equal to the size of every line through $(\theta)$. 

\def\calD{{\cal D}}
\def\mcalD{-{\cal D}}
\def\bezer{{\setminus \{ 0 \}}}
  Remind also that for any point 
   $(x)$ and line $[y]$ of ${\goth P}^n$:
\begin{equation}\label{wz:relofinc}
   (x)\inc [y]  \textrm{ iff } 
    \begin{array}{l}
      x_i-y_i\in {\cal D} \textrm{ and }
      \\
      x_j=y_j  \textrm{ for all } j\neq i, j\in \{1,\ldots,n\}
    \end{array}
      \textrm{ for some }
     i\in \{1,\ldots,n\}. 
\end{equation}
  Recall 
  (cf. \ref{lem:analchain} \eqref{analchain:cas2} and  
  \ref{lem:analchain}\eqref{analchain:cas3})
  that for $x,x' \in (C_k)^n$
  \begin{equation}\label{wz:relocol}
    (x),(x') \textrm{ are collinear } \iff [x],[x'] \textrm{ intersect} \iff x -x' \in D - D
  \end{equation}
  Note that $|\supp(u)|\leq 2$ for every $u \in D - D$.
  Moreover, if $|\supp(u)| = 1$ then $u_i \in \calD - \calD = C_k$
  for $i \in \supp(u)$ and if $|\supp(u)| = 2$ then 
  $u \in \Xi^2_{\{\alpha,\beta\}}$, where 
  $\alpha \in \calD \setminus \{ 0 \}$ and 
  $\beta \in \mcalD \setminus \{ 0 \}$.

  Now, we establish some crucial facts.
The first  is immediately from \eqref{wz:relofinc}.
\begin{lem}\label{line:through:0}
  The line $[y]$ passes through $(\theta)$ iff 
  $[y]\in\flines_i$ for some $i$ and $y_i \in \mcalD$.
\end{lem}
Analogously, the following is immediate from \eqref{wz:relocol}.
\begin{lem}\label{pty:otocz}
  Let $x \in (C_k)^n$. The point $(x)$ is a point of $\otocz({\goth P}^n,{(\theta)})$
  iff $x \in P_i$ for some $i$ (i.e. $|\supp(x)| = 1$) or
  $x\in \Xi_{\{\alpha,\beta\}}^{2}$ and
  $\alpha \in \mcalD \bezer$,
  $\beta \in \calD \bezer$.
\end{lem}
\begin{lem}\label{clofan}
  If $x \in P_i$ then either the line $[x]$ passes through $(\theta)$ and its size in 
  $\otocz({\goth P}^n,{(\theta)})$ is $qn+1$, or the size of $[x]$
  in $\otocz({\goth P}^n,{(\theta)})$ is $q+1$; then, in particular, $[x]$ does not contain
  any point of $P_j$ with $j \neq i$.
\end{lem}

\begin{lem}\label{kroski}
  Let $x\in(C_7)^n$. If the line $[x]$ contains a point of $\otocz({\goth P}^n,{(\theta)})$
  (i.e. it intersects a line of the form $[y]$ 
  defined in \ref{line:through:0})
  then $|\supp(x)| \leq 3$. Moreover, if $|\supp(x)| = 3$ then
  the size of $[x]$ in $\otocz({\goth P}^n,{(\theta)})$ is $2$.
\end{lem}
\begin{proof}
  Assume that $[x]$ and $[y]$, as above, have common point $(z)$.
  Then $|(\supp({x-y})| \leq 2$. From assumption, $|\supp(y)|=1$
  and thus $|\supp(x)|\leq 3$.
  Let $y_{i_1} \neq 0$. If $|\supp(x)| = 3$ then $x_{i_1} = y_{i_1} \in \mcalD$.
  Write $\supp(x) = \{ i_1,i_2,i_3 \}$. 
  Since 
  $x_{i_1} - y_{i_1} = 0$, $x_{i_2} = x_{i_2}-y_{i_2}$, and $x_{i_3} = x_{i_3} - y_{i_3}$,
  the condition $x-y\in D-D$ gives the description of $x$.
  Consequently, $x_{i_2}\in \mcalD$ and $x_{i_3} \in \calD$
  (or symmetrically, with $i_2,i_3$ interchanged).
  It is seen that among the lines through $(\theta)$ 
  only $[y]$ and $[y']$ are crossed by $[x]$ where 
  $\supp(y') = \{ i_2 \}$,
  $y'_{i_2} = x_{i_2}$.
\end{proof}
There are lines in $\otocz({\goth P}^n,{(\theta)})$ linking points in $P_i$ with
points in $\Xi_{\{\alpha,\beta\}}^{2}$ where $\alpha \in \mcalD \bezer$, $\beta \in \calD \bezer$. Namely:
\begin{auxlem}\label{singer:jointriangle}
 Let   $(x)\in S_i$, $[y]\in \flines_i$, and $(x)\inc [y]$.
 For every $j\neq i$ there are $q$ lines of the 
size $2$ in 
$\otocz({\goth N}^n,{(\theta)})$, such that each of them joins $(x)$ with one of the 
pairwise collinear points
$(z^1),\ldots (z^q)$, where     
$z^1_i=\ldots=z^q_i=x_i-y_i$, $\{z^1_j,\ldots,z^q_j\}= -\calD\setminus \{0\}$, and
$z^1_s=\ldots=z^q_s=0$ for all $s\neq i,j$; $s = 1,\ldots,n$.
\end{auxlem}
\begin{auproof}
Assume $x\in S_i$, $[y]\in \flines_i$, and $(x)\inc [y]$ for some
$i\in\{1,\ldots,n\}$. Take the point $(z)$ with $z_i=x_i-y_i$, $z_j\in -\calD\setminus \{0\}$ and
$z_s=0$ for all $s\neq i,j$. Then, from \eqref{wz:relofinc}, $x_i-y_i\in \calD$, so
$x_i\in \calD + y_i$. One can note that  $z_i\in\calD$, and thus
$x-z\in \calD-\calD$. 
Moreover, $(z)\in \otocz({\goth N}^n,{(\theta)})$ -- in view of 
\ref{pty:otocz}.
Now, the claim follows directly from \eqref{wz:relocol}. 
\end{auproof}
\begin{auxlem}\label{connect:fanem}
  For any two points $a,a'$ of $\goth M$ there is a sequence $b_0,\ldots,b_m$ of points
  of $\goth M$ such that $b_0 = a$, $b_m = a'$, and 
  $b_j$ is a point of a cyclic projective subplane 
  in ${\otocz({\goth P}^n,{b_{j-1}})}$ for $j = 1,\ldots,m$.
\end{auxlem}
\begin{auproof}
  Without loss of generality we consider $a = (\theta)$; let $a' = ((a'_1,\ldots,a'_n))$.
  Take the sequence 
  $\big( u_j = (u_{j,1},\ldots,u_{j,n}) \colon j = 1,\ldots,n \big)$ 
  of elements of $(C_k)^n$ defined by 
  \begin{ctext}
    $u_{j,i} = 0$ for $i\neq j$, $u_{j,j} = a'_j$
  \end{ctext}
  and put $b_0 = a$,
  $b_j = (\tau_{u_j} \circ \tau_{u_{j-1}} \circ \ldots \circ \tau_{u_1})(a)$ 
  for $j =1,\ldots,n$.
\end{auproof}

In the sequel we shall frequently determine the number of solutions of the following
problems:
\def\casal{{\ensuremath{(\alpha)}}}
\def\casbe{{\ensuremath{(\beta)}}}
\begin{ctext}
  given $y \in C_k$, $y\neq 0$ determine $u \in \mcalD$ such that 
  $\begin{array}{ll}
  \casal: & y-u \in \calD \\ \casbe: & y-u \in \mcalD.
  \end{array}$
\end{ctext}
\begin{lem}\label{soluty}
  The problem \casal \ has exactly one solution.
  The problem \casbe \ has exactly one solution when $y \in -2\calD$,
  in the other case it has either two distinct solutions,
  or it has no solution.
\end{lem}
\begin{proof}
  Write $u = -d$, with $d \in \calD$.
  In the case \casal \ we search for $d' \in \calD$ with
  $y - u = d'$ i.e. $y = d' - d$. Since $y \neq 0$, both $d$ and $d'$ are 
  uniquely determined by $y$.

  \par
  Let us pass to \casbe. We need 
  $y - u = -d'$ for some $d' \in \calD$; this means $y = -d - d'$.
  Assume that $-d - d'  = -d_1 - d'_1$.
  Then $d_1 - d' = d - d'_1$ and eiher $d_1 = d'$, $d = d'_1$,
  or $d_1 = d$, $d = d'_1$.
  The two solutions $u=-d$ and $u=-d'$ of \casbe \ coincide iff $y = - 2d \in -2\calD$.
\end{proof}

Recall that $\calD \cap \mcalD = \{ 0 \}$; indeed, suppose that
$d_1 = -d_2$ for some $d_1,d_2 \in \calD$.
Then we write $d_1 - 0 = 0 -d_2$ -- the rest is evident from the definition
of a quasi difference set.

Now, we pass to determining lines $[y]$ of $\otocz({\goth M},{(\theta)})$ of the size 
at least $3$; in view of \ref{kroski} we can assume that
$|\supp(y)| = 2$ and then, 
to simplify formulas without loss of generality we assume that
$\supp(y) = \{ 1,2 \}$.
Let $[u]$ thorugh $(\theta)$ cross $[y]$; 
in view of \eqref{wz:relocol}, $\supp(u) \subset \supp(y)$ and
there are two cases to consider:
$\supp(u) = \{ 1 \}$
or
$\supp(u) = \{ 2 \}$.
Let us start with the first one.
Remember that 
  $$u = - d_1 \text{ and } d_1 \in \calD, \quad \text{and}\quad y_1,y_2 \neq 0.$$
There are three possibilites:
\begin{enumerate}[(a)]\itemsep-2pt
\item\label{moze:a}
  $y_1 - u_1 = 0$ and then $y_2 \in C_k \bezer$ is arbitrary;
\item\label{moze:b}
  $y_1 - u_1 \in \calD \bezer$, $y_2 \in \mcalD \bezer$;
\item\label{moze:c}
  $y_1 - u_1 \in \mcalD \bezer$, $y_2 \in \calD \bezer$.
\end{enumerate}
Assume \eqref{moze:a}. 
We try to find another line $[u']$ through $(\theta)$
that crosses $[y]$.
Assume that $u'_1 \neq 0$. 
If $y_1 - u'_1 = 0$ then $u =u'$.
If $y_1 - u'_1 \in \calD \bezer$ from \ref{soluty} we come to  $u_1 = u'_1$ 
(note: $y_1 - u_1 \in \calD$!).
If $y_1 - u'_1 \in \mcalD \bezer$ write $y_1 - u_1 = -0 \in \mcalD$; 
from \ref{soluty} two possible distinct solutions of the corresponding \casbe \
are $u_1$ or $0$.
Again, this way we cannot obtain $u'\neq u$.
\\
Now, we assume that $u'_2 \neq 0$.
The line $[u']$ crosses $[y]$ in three cases:
\begin{enumerate}[({\ref{moze:a}.}1)]\itemsep-2pt
\item\label{abo:1}
  $y_2 - u'_2 = 0$; then $y_1$ is arbitrary $\neq 0$
  (though we already know that $y_1 \in \mcalD$, since $y_1 = u_1 = -d_1$);
\item\label{abo:2}
  $y_2 - u'_2 \in \calD\bezer$; then $y_1 \in \mcalD\bezer$ which actually is valid.
\item\label{abo:3}
  $y_2 - u'_2 \in \mcalD\bezer$ and $y_1 \in \calD\bezer$, which is impossible,
  since $y_1 \in \mcalD$.
\end{enumerate}
In case (\ref{moze:a}.\ref{abo:1}), $u'_2$ is determined uniquely, but then necessarily
$y_2 \in \mcalD$. 
Moreover, $y_2 - u'_2 \in \calD$,
and if so, there is no $u''$ with $u''_2 \neq u'_2,0$ that may satisfy
$y_2 - u''_2 \in \calD\bezer$.
Finally, if $y_1,y_2 \in \mcalD\bezer$, the size of $[y]$ is $2$.
\newline
In case (\ref{moze:a}.\ref{abo:2}), from \ref{soluty} we find that there is
exactly one possible $u'_2$.
Summing up, we get that the size of $[y]$ is at most $2$.

Assume the case \eqref{moze:b} or \eqref{moze:c}; moreover, assume that $y_1\notin\mcalD$,
since the case where $y_1 \in \mcalD$ was already examined in \eqref{moze:a}. 
By the symmetry, we can also assume that
$y_2 \notin\mcalD$.
Thus it remains to consider the case \eqref{moze:c} only;
then $y_2 \in \calD\bezer$.
From \ref{soluty}, there are at most two $u'\in\mcalD\bezer$ such that
$u'_1 \neq 0,u_1$ and $y_1 - u'_1 \in \mcalD\bezer$
(to this aim we need $y_1 \in -(\calD + \calD)$ but $y_1 \notin -2\calD$).

If we want to have another line $[u']$ passing through $(\theta)$
and intersecting $[y]$
we must have $u'_2\in\mcalD\bezer$.
If $y_2 - u'_2 \in \calD\bezer$, then $y_1 \in \mcalD\bezer$, which contradicts
assumptions.
Therefore, $y_2 - u'_2 \in\mcalD\bezer$ and $y_1 \in \calD\bezer$.
Again, there are at most two such $u'_2$; there are exactly two
when $y_2 \in -(\calD + \calD)$ but $y_2 \notin -2\calD$.

The size of the line $[y]$ varies from $1$ to $4$.
It is equal to $4$
in the case where
\begin{equation}\label{war:size:4}
  y_i \in (-(\calD + \calD))\cap\calD \setminus (-2\calD \cup \mcalD)
  \text{ for } i=1,2.
\end{equation}
The size of $[y]$ is $3$   
if $y_{i_1}$ as in \eqref{war:size:4}, and 
$y_{i_2}$ satisifes
\begin{equation}\label{war:size:3}
   y_{i_2} \in (-(\calD + \calD))\cap\calD \cap (-2\calD) \setminus \mcalD.
\end{equation}
with $\{ i_1,i_2 \} = \{ 1,2 \}$.


\subsubsection{A power of Fano plane}

Now, let us put 
  $\goth M = {\goth F}^n,
  \textrm{ where } {\goth F}=\DifSpace(C_7,\{0,1,3\})$.
Our goal is to determine the automorphisms group
of ${\goth F}^n$. We claim as follows:
\begin{figure}[!h]
    \begin{center}
    \includegraphics[scale=0.7]{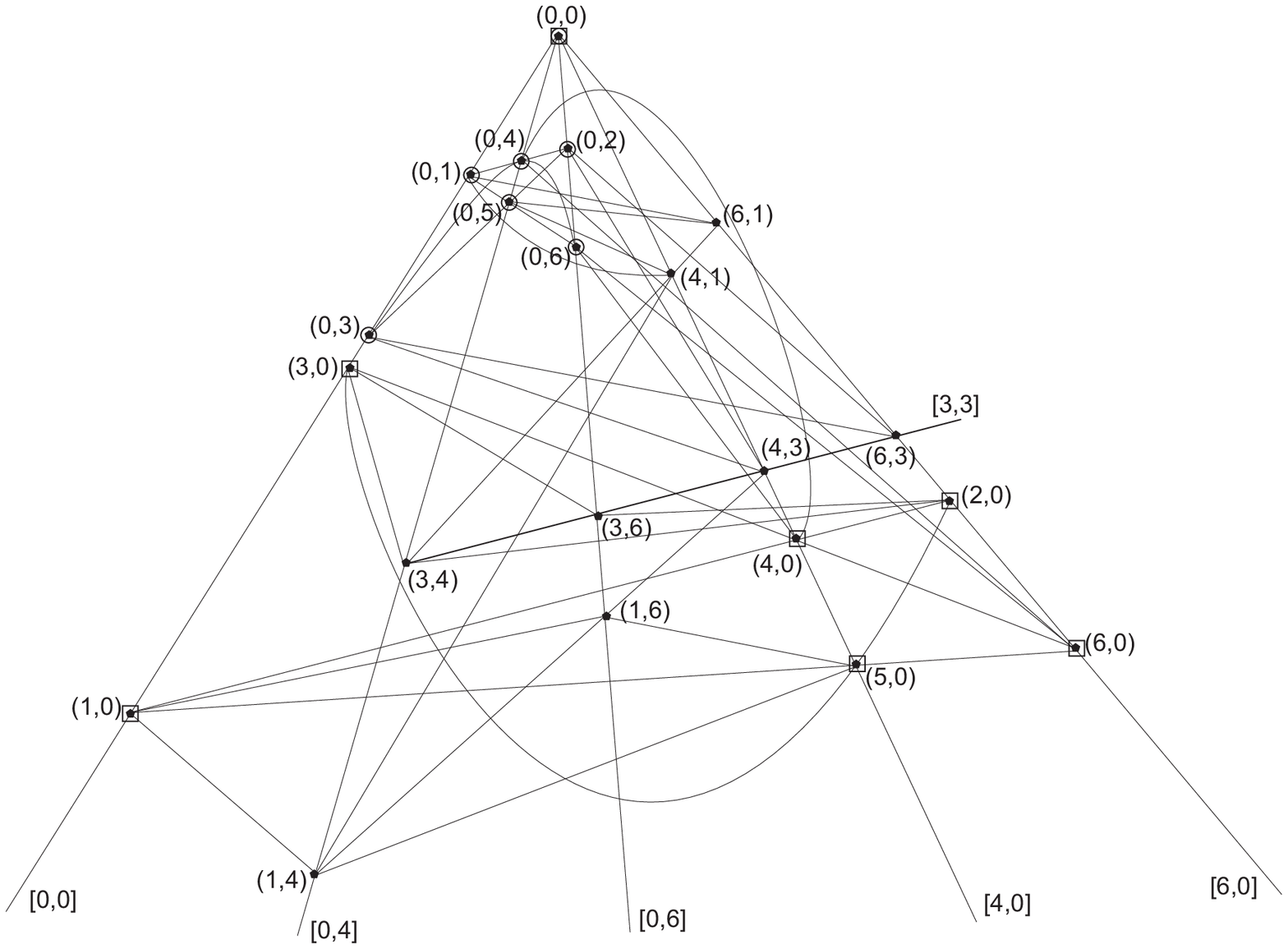}
    \end{center}
\caption{Neighborhood of the point $(0,0)$ in $\goth F \oplus \goth F$}
\label{fig:FFanoneighb}
\end{figure} 

\begin{prop}
  Let ${\goth P}^n$ be the structure defined in \eqref{multisinger} with 
  ${\goth P}:={\goth F}=\DifSpace(C_7,\{0,1,3\})$. 
  Then, the group $\Aut({\goth F}^n)$
  is isomorphic to $  S_n \ltimes (C_7)^n $.
\end{prop}
\begin{proof}
  Let $(0,\ldots,0)=:\theta$, and 
    $\tau_{(u_1,\ldots,u_n)}((x_1,\ldots,x_n))=(x_1+u_1,\ldots,x_n+u_n)$  
   $\textrm{ for } (u_1,\ldots,u_n),(x_1,\ldots,x_n)\in (C_7)^n$.
  Take under consideration 
    $F=\tau_{-f(\theta)}\circ f$, where $f\in \Aut({\goth F}^n)$.
  Clearly, 
    $\tau_{(u_1,\ldots,u_n)}\in \Aut({\goth F}^n)$, 
  and consequently
    $F\in \otocz({\Aut({\goth F}^n)},{(\theta)})$.
\def\linesiii{\flines'_0}
\def\linesiv{\flines''_0}
  To understand better the geometry of ${\goth F}^n$ we
  analyze $\otocz({\goth F}^n,{(\theta)})$  
  (see Figure \ref{fig:FFanoneighb} for the case $n=2$).
\newline
  Let $\linesiv$ be the family of lines of the size $4$ in $\otocz({\goth F}^n,{(\theta)})$
  and $\linesiii$ be the family of the lines of the size $3$ in $\otocz({\goth F}^n,{(\theta)})$
  that are not in any of the $\flines_i$. 
  To find a form of these lines we apply the 
  obtained conditions \eqref{war:size:4} and \eqref{war:size:3}
  to the classical Fano plane. 
  We compute as follows:
  \begin{itemize}\def\labelitemi{\strut}\itemsep-2pt
  \item
    $\mcalD = \{ 0,4,6 \}$;
  \item
    $2\calD  = \{ 0,2,6 \}$; then $-2\calD = \{ 0,1,5 \}$
  \item
    $-2\calD \cup \mcalD = \{ 0,1,4,5,6 \}$;
  \item
    $\calD + \calD = \{ 0,1,2,3,4,6 \}$,
    then $-(\calD +\calD) = \{ 0,1,3,4,5,6 \}$
  \item
    $-(\calD + \calD) \cap \calD = \{ 0,1,3 \}$
  \item
    $-(\calD + \calD) \cap \calD \cap -2\calD = \{ 0,1 \}$.
\end{itemize}
Next we can infer:
\begin{auxlem}\label{line:size:4}
  Let $y \in (C_7)^n$.
  Then $[y]\in\linesiv$ iff $y\in \Xi_{3}^{2}$.
  \par
  Let $y\in \Xi_{3}^{2}$ and $\supp(y) = \{ i_1,i_2 \}$.
  Clearly, $[y]$ does not intersect $[\theta]$, but 
  $[y]$ intersects every of the remaining
  two lines in $\flines_{i_1}$ and in $\flines_{i_2}$.
  Consequently, for any two $i_1,i_2 \in \{ 1,\ldots,n \}$
  there is (exactly one) line in $\linesiv$ that crosses two lines in
  $\flines_{i_1}$ and two lines in $\flines_{i_2}$.
  \par
  No two distinct lines in $\linesiv$ intersect.
\end{auxlem}
\begin{auproof}
 Suffice to see that in this case \eqref{war:size:4} defines the set
  \begin{itemize}\def\labelitemi{\strut}\itemsep-2pt
      \item
      \strut\quad
      $(-(\calD + \calD))\cap\calD \setminus (-2\calD \cup \mcalD) =
    \{ 0,1,3 \} \setminus \{ 0,1,4,5,6 \} = \{ 3 \}$,
  \end{itemize}
  that together with $|\supp(y)|=2$ produces the first claim.
  \par
  To prove the second claim note that if $\supp(y) = \{ i_1,i_2 \}$ 
  then the common points of $[y]$ and lines in 
  $\flines_{i_1}$ are $x,x'$ with $|\supp(x)| = 2$ such that
  $x_{i_1} \in \{ 4,6 \} $ and $x_{i_2} = 3$  
  and similarly for $i_2$.
\end{auproof}
\begin{auxlem}\label{line:size:3}
  If $y\neq\theta$, then $[y]$ is of the size $3$ in 
  $\otocz({\goth F}^n,{(\theta)})$ iff either $y\in \Xi_{\{1,3\}}^{2}$ or
  $y\in P_i$ for some $i\in\{1,\ldots,n\}$. This gives, in particular, that
  $\linesiii = \{ [y]\colon y \in \Xi_{\{1,3\}}^{2} \}$.
 \par
  Let $\supp(y) = \{ i_1,i_2 \}$, $y_{i_1} = 1$, and $y_{i_2} = 3$.
  Clearly, $[y]$ and $[\theta]$ do not intersect.
  The line $[y]$ crosses two other lines in $\flines_{i_2}$
  and it crosses exactly one line in $\flines_{i_1}$.
  Consequently, for every two $i_1,i_2 \in \{ 1,\ldots,n \}$
  there is (exactly one) line in $\linesiii$ hat crosses two lines in $\flines_{i_2}$
  and crosses exactly one line in $\flines_{i_1}$.
 \par 
  No two distinct lines in $\linesiii$ intersect.
\end{auxlem}
\begin{auproof}
As in proof of \ref{line:size:4} we find the set $\{3\}$ defined by \eqref{war:size:4}
and substitute for \eqref{war:size:3} as follows:
\begin{itemize}\def\labelitemi{\strut}\itemsep-2pt
   \item
      \strut\quad
      $(-(\calD + \calD))\cap\calD \cap (-2\calD) \setminus \mcalD
    = \{ 0,1 \} \setminus \{ 0,4,6 \} = \{ 1 \}$.
    \end{itemize}
\par
  To close the proof note that if $y$ is as required, then
  $[y]$ crosses $[y'],[y'']\in\flines_{i_2}$, where
  $y'_{i_2} = 4$ and $y''_{i_2} = 6$ in 
  the points $(x')$ and $(x'')$ resp., such that $x'_{i_2} = 4$, $x''_{i_2} = 6$,
  $x'_{i_1} = 1 = x''_{i_1}$.
  The line $[y]$ crosses exactly one line in $\flines_{i_1}$;
  namely the line $[y']$ such that $y'_{i_1} = 4$
  in the point $(x')$ such that $x'_{i_1} = 4$ and $x'_{i_2} = 3$.
\end{auproof}
  Directly from \ref{line:size:3} we get
\begin{auxlem}\label{lines:iii}
  A line $L$ in $\otocz({\goth F}^n,{(\theta)})$ belongs to $\linesiii$
  iff it is of the size $3$ and no other line of the size $3$ in $\otocz({\goth F}^n,{(\theta)})$
  crosses $L$.
\end{auxlem}
  Observing \ref{line:size:3}, \ref{line:size:4} 
  (and their proofs: formulas for corresponing points of intersection), 
  and \ref{pty:otocz} we obtain immediately
\begin{auxlem}\label{points:fano}
  Let $(x)$ be a point of $\otocz({\goth F}^n,{(\theta)})$.
  Then $x \in P_i$ iff there are two distinct lines of the size $3$ that pass through it.
  The set of points on the lines in $\linesiii \cup \linesiv$ is exactly
  the set of points of the form $(x)$ with $x \notin \cup_{i=1}^{n}P_i$.
\end{auxlem}
  From the above analysis it follows that $\otocz({\goth F}^n,{(\theta)})$
  contains exactly $n$ subconfigurations isomorphic to a Fano plane; these are
  exactly substructures of the form 
    $\struct{S_i,\flines_i,\inc}$.
  Intuitively, we can read  \ref{line:size:4} as a statement
  that any two Fano subplanes of $\otocz({\goth F}^n,{(\theta)})$
  are joined by a line of the size $4$.
  Analogously, \ref{line:size:3} explains how lines of the size $3$ join
  the above Fano subplanes.

  In view of \ref{points:fano} and \ref{lines:iii} our automorphism $F$ preserves 
  the set $\bigcup_{i=1}^n S_i$ and further, it permutes the above Fano
  subplanes (clearly, a Fano subplane must be mapped onto a Fano subplane).
  That gives that $F$ determines a permutation $\sigma$ such that
  $F$ maps the set $S_i$ onto $S_{\sigma(i)}$ and it 
  maps the family $\flines_i$ onto $\flines_{\sigma(i)}$
  for every $i =1,\ldots,n$.

  Obviously, $F$ preserves the set of lines of the size $4$ in 
  $\otocz({\goth F}^n,{(\theta)})$. Since these lines are of the form $[y]$ with 
  $y\in \Xi_{3}^{2}$, we can identify every such a line $[y]$ with the set 
    $\supp(y) \in \sub_2(\{ 1,\ldots,n \})$.
  Every point $(x)$, where $x\in \Xi_3^3$, is in ${\goth F}^n$ the
  meet of three lines $[y_t]$, $y_t\in\Xi_3^2$, and $t=1,2,3$  iff
  $\supp(y_t)\subset\supp(x)$. Therefore, lines
  $\{[y]\colon y\in\Xi_3^2\}$ together with their intersection points form the  
  structure dual to $\gras(3,n)$. 
  The map $F$ determines a permutation $F_0$ of the lines in $\linesiv$
  which, in view of the above, is an automorphism of $\gras(3,n)$.
  The automorphisms group of $\gras(3,n)$ 
  is the group $S_n$ -- compare \cite{corset}
  and thus there is $\sigma'\in S_n$ which determines $F_0$.
  It is seen (cf. \ref{line:size:4}) that $\sigma' = \sigma$.
  Let $G = G_{\sigma}$ be the automorphism of ${\goth F}^n$ 
  determined by the permutation $\sigma$ (cf. \ref{lem:syminsegre}) and
  let $\varphi=G^{-1}\circ F$. 
  Clearly, $\varphi$ is an automorphism of
  ${\goth F}^n$, and $\varphi$ maps every line in $\linesiv$ onto itself.
  Consequently, $\varphi$ maps every family 
    ${\cal J}_i \setminus \{ [\theta] \}$ 
  onto itself and thus it leaves the line $[\theta]$ invariant.

  From \ref{lines:iii},
  the map $\varphi$  preserves the family $\linesiii$; 
  observing intersections
  of the lines of this family and the lines in the families $\flines_i$ 
  (cf. \ref{line:size:3})
  we get that every line through $(\theta)$ remains invariant under $\varphi$.
  \par
  Now, we need three other global properties.
\begin{auxlem}\label{au:1}
  Let $F$ be an automorphism of $\goth M$ such that $F$ leaves every line 
  through a point $a$ invariant.
  Then 
  $F\restriction{\otocz({\goth F}^n,{a})}=\id$.
\end{auxlem}
\begin{auproof}
  Without loss of generality we can assume that $a = \theta$
  (consider $F':= \tau_{-a} \circ F \circ \tau_a$, if necessary) and then we can apply 
  characterizations proved before.
 \par
  From \ref{line:size:4} we get that $F$ fixes every point on an arbitrary line in
  $\linesiv$ and it fixes every point on an arbitrary line in $\linesiii$.
  In view of \ref{points:fano} this gives that $F$ fixes every point outside
  the $S_i$.
  Considering the set of lines of the size $2$ in $\otocz({\goth F}^n,{(\theta)})$
  (compare Figure \ref{fig:FFanoneighb})
  we conclude that every point $(x)$ with $x\in P_i$ is fixed as well
  which closes the proof.
\end{auproof}
\begin{auxlem}\label{au:2}
  Let $a,b$ be two points of ${\goth F}^n$ such that 
  $b$ is a point of a Fano subconfiguration 
  in ${\otocz({\goth F}^n,{a})}$.
  If $F$ is an automorphism of ${\goth F}^n$ such that
  $F\restriction{\otocz({\goth F}^n,{a})}=\id$
  then
  $F\restriction{\otocz({\goth F}^n,{b})}=\id$ as well.
\end{auxlem}
\begin{auproof}
  Again, we assume that $a = (\theta)$.
  The degree of the point $b$ in
  $\otocz({\goth F}^n,{(\theta)})$ amount to $2n+1$, so every line through $b$ is
  preserved and the claim follows from \ref{au:1}.
\end{auproof}
  Now, we return to the proof of the theorem.
  Recall that we have already proved that $\varphi$ preserves every line
  through $(\theta)$. From  \ref{au:2} and \ref{connect:fanem} 
  by a straightforward induction we come to $\varphi = \id$ and thus
  $F = G_{\sigma}$ and $f = \tau_{f(\theta)} \circ G_{\sigma}$.
 \par
  To close the proof, note that for elements
  $x = (x_1,\ldots,x_n)$, $u = (u_1,\ldots,u_n)$ of $(C_7)^n$  
  we have 
\begin{multline}\nonumber
  G_{\sigma} \circ \tau_u ((x)) = 
  G_{\sigma}((x_1+u_1,\ldots,x_n+u_n)) =
  ((x_{\sigma(1)}+u_{\sigma(1)},\ldots,x_{\sigma(n)}+u_{\sigma(n)})) =
  \\
  = G_\sigma(x_1,\ldots,x_n)+ G_\sigma(u_1,\ldots,u_n) =
    \tau_{G_\sigma((u_1,\ldots,u_n))}\circ G_\sigma((x_1,\ldots,x_n)),
\end{multline}
  i.e. $G_\sigma \circ \tau_u \circ G_{\sigma}^{-1} = \tau_{G_\sigma(u)}$, as required.
\end{proof}


\subsubsection{A power of cyclic projective plane $PG(2,3)$}

Let us adopt 
  $\goth M = {\goth N}^n,
  \textrm{ where } {\goth N}=\DifSpace(C_{13},\{0,1,3,9\})$ -- cyclic projective plane
  $PG(2,3)$.

Note that the map 
   $$\mu_\alpha\colon C_{13} \ni x \longmapsto \alpha\cdot x \quad \text{with} \quad
   \alpha = 3$$
 is an automorphism of the group $C_{13}$ that leaves the set $\calD$ invariant,
 and therefore it determines an automorphism of $\goth N$. 
 Moreover for the induced automorphism we have 
 $\mu_3[y] = [\mu_3(y)]$ for every $y \in C_{13}$ (cf. \ref{rem:whenautline}). 
 Clearly, $\mu_3$ generates the $C_3$ group 
 \newline
 \begin{equation}\label{wz:graut:C13}
 H:= \{ \mu_1 = \id,\mu_3,\mu_9 \} \subset \otocz({\Aut({\goth N})},{(\theta)}).
 \end{equation}
 We establish the automorphisms group of $\goth M$.
\begin{prop}
  Let ${\goth P}^n$ be the structure defined in \eqref{multisinger} with 
  ${\goth P}:={\goth N}=\DifSpace(C_{13},\{0,1,3,9\})$. 
  Then, the group $\Aut({\goth N}^n)$
  is isomorphic to $S_n\ltimes ( (C_3)^n \ltimes (C_{13})^n ) $.
\end{prop}
\begin{proof}
  Let $(0,\ldots,0)=:\theta$, and 
    $\tau_{(u_1,\ldots,u_n)}((x_1,\ldots,x_n))=(x_1+u_1,\ldots,x_n+u_n)$  
   $\textrm{ for } (u_1,\ldots,u_n),(x_1,\ldots,x_n)\in (C_{13})^n$.
  Define
    $F:=\tau_{-f(\theta)}\circ f$, where $f\in \Aut({\goth N}^n)$.
  Obviously, 
    $\tau_{(u_1,\ldots,u_n)}\in \Aut({\goth N}^n)$ 
  and consequently
    $F\in \otocz({\Aut({\goth N}^n)},{(\theta)})$.
\par
 For $q=3$, $k=13$, and $\calD = \{ 0,1,3,9 \}$ lemmas \ref{line:through:0}, 
 \ref{pty:otocz}, \ref{clofan}, \ref{kroski} give us some description 
 of $\otocz({{\goth N}^n},{(\theta)})$. Now, we examine the structure
 of lines of the size $3$ and $4$ in $\otocz({{\goth N}^n},{(\theta)})$.
 Thus, we apply \eqref{war:size:3} and \eqref{war:size:4} to $\calD = \{ 0,1,3,9 \}$
over $C_{13}$:
\begin{itemize}\def\labelitemi{\strut}\itemsep-2pt
\item
  $\mcalD = \{ 0,4,10,12 \}$;
\item
  $2\calD  = \{ 0,2,6,5 \}$; then $-2\calD = \{ 0,11,7,8 \}$
\item
  $-2\calD \cup \mcalD = \{ 0,4,7,8,10,11,12 \}$;
\item
  $\calD + \calD = \{ 0,1,2,3,4,5,6,9,10,12 \}$,
  then $-(\calD +\calD) = \{ 0,1,3,4,7,8,9,10,11,12 \}$
\item
  $-(\calD + \calD) \cap \calD = \{ 0,1,3,9 \}$
\item
  $-(\calD + \calD) \cap \calD \cap -2\calD = \{ 0 \}$
\item
  \eqref{war:size:4} defines the set
  \\
  \strut\quad
  $(-(\calD + \calD))\cap\calD \setminus (-2\calD \cup \mcalD) =
  \{ 0,1,3,9 \} \setminus \{ 0,4,7,8,10,11,12 \} = \{ 1,3,9 \}$
  \\
  and \eqref{war:size:3} defines 
  \\
  \strut\quad
  $(-(\calD + \calD))\cap\calD \cap (-2\calD) \setminus \mcalD = \emptyset$.
\end{itemize}
\def\linesiii{\flines'_0}
\def\linesiv{\flines''_0}
Let $\linesiv$ be the family of lines of the size $4$ in $\otocz({\goth N}^n,{(\theta)})$
that are not in any of the $\flines_i$,
and $\linesiii$ be the family of the lines of the size $3$ in $\otocz({\goth N}^n,{(\theta)})$. 
\newline
Straightforward inference from the above computation is the following:
\begin{auxlem}\label{singer:size:4}
Let $y \in (C_{13})^n$.
  Then $[y]\in\linesiv$ iff $y\in \Xi_{\{\alpha,\beta\}}^{2}$, where
  $\alpha,\beta \in \{ 1,3,9 \}$.
  \par
  If besides $\supp(y) = \{ i_1,i_2 \}$ then
  $[y]$ intersects two out of 
  three lines in $\flines_{i_1}$ and in $\flines_{i_2}$, but do not intersects $[\theta]$.
  Consequently, counting all such possibilities, for any two $i_1,i_2 \in \{ 1,\ldots,n \}$
  there are exactly nine lines in $\linesiv$ that cross two lines in
  $\flines_{i_1}$ and two lines in $\flines_{i_2}$. 
  \par
  For every point $(x)\in \otocz({\goth N}^n,{(\theta)})$ with $x\notin P_i$
  there are two lines $[y]\in \linesiv$ such that
  $(x) \inc [y]$.
  \par
  Every line in $\linesiv$ intersects four other lines in $\linesiv$ and does not intersect
  remaining four lines from $\linesiv$.
\end{auxlem}
\begin{auproof}
Let $\supp(y) = \{ i_1,i_2 \}$ for $[y]\in\linesiv$. From 
\ref{line:through:0} and \ref{pty:otocz}, 
  the intersection points of $[y]$ and lines in 
  $\flines_{i_1}$ are such $(x)$ that $\supp(x) = \{ i_1,i_2 \}$ and $x_{i_2}\in \{ 1,3,9\}$.
  What is more:
  \begin{itemize}\def\labelitemi{\strut}\itemsep-2pt
  \item
   if $x_{i_2}= 1$ then $x_{i_1}\in \{4,10\}$,
  \item
   if $x_{i_2}= 3$ then $x_{i_1}\in \{4,12\}$,
   \item
    if $x_{i_2}= 9$ then $x_{i_1}\in \{10,12\}$.
\end{itemize}
  Analogous computation we can do for $\flines_{i_2}$.
  \par
  On the other hand, we consider $(x)\in \otocz({\goth N}^n,{(\theta)})$ with $\supp(x) = \{ i_1,i_2 \}$ and
  $[y]\in\linesiv$ with $\supp(x) = \supp(y)$.
  From \ref{pty:otocz} $x_{i_1}\in \{4,10,12\}$ and $x_{i_2}\in \{1,3,9\}$. 
  For every element $\alpha\in \{4,10,12\}$ there exist two elements 
  $\beta_1,\beta_2\in \{1,3,9\}$ such that $\alpha-\beta_1,\alpha-\beta_2\in \calD$; and then
  \eqref{wz:relofinc} justifies the next part of our claim. 
  The last part follows immediately from \eqref{wz:relocol}.
\end{auproof}
\begin{figure}[!h]
    \begin{center}
    \includegraphics[scale=0.7]{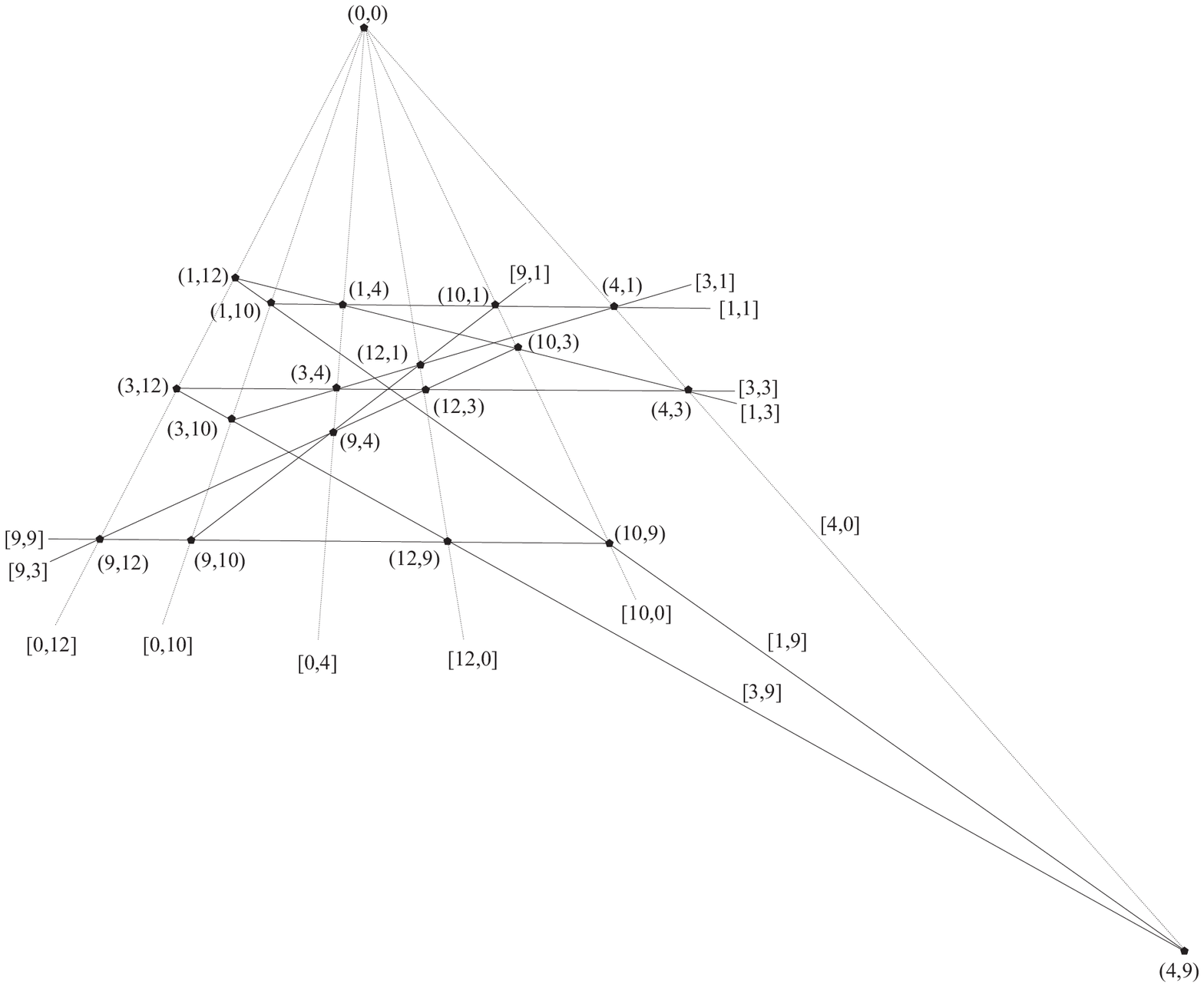}
    \end{center}
\caption{The structure of lines of the size $4$ in $\otocz({\goth N}^2,{(\theta)})$} 
\label{fig:4sizelines}
\end{figure} 
\begin{auxlem}\label{singer:size:3}
 There are no lines of the size $3$ in $\otocz({\goth N}^n,{(\theta)})$ (i.e. $\linesiii=\emptyset$).
\end{auxlem}
 %
%
Let us consider the set 
\newline
\centerline{
$\flines^{\alpha}:=\{[y]\in \flines_i\colon y_i=\alpha\in -\calD \textrm{ and } i=1,\ldots,n\}$.} 
We claim the following:
\begin{auxlem}\label{singer:fourpreserve}
If $F$ is an automorphism of ${\goth N}^n$  leaving every line in $\flines^{\alpha}$ invariant
 then $F\restriction{\otocz({\goth N}^n,{(\theta)})}=\id$.
\end{auxlem}
\begin{auproof}
Without loss of generality we can take $\alpha=4$.
Assume that $F\in\Aut({\goth N}^n)$ and $F$ leaves every line in $\flines^4$ invariant.
For arbitrary two distinct $i,j\in \{1,\ldots,n\}$ observe $[y'],[y'']\in \flines^4$ such that
$\supp(y')=i$, $\supp(y'')=j$. 
Note that four lines of the form $[z]$ with $z\in \Xi_{\{\alpha,\beta\}}^{2}$, where
$\alpha,\beta\in \{1,3\}$ and $\supp(z)=\{i,j\}$ are in view of \ref{singer:size:4}
unique lines of the size $4$ crossing both $[y']$ and $[y'']$. 
Consequently, $F$ must preserve such family of lines. 
\par
Let the line $[z]$ with 
$z_i=z_j=1$ not be mapped onto itself. In every of three such cases, two
triangles $(x),(x'),(x'')$ with   
\newline
\centerline{
$x_i=x'_i=x''_i\in \{1,3\},\; \{x_j,x'_j,x''_j\}= -\calD\setminus \{0\}$}
(or symmetric case for $i$ and $j$)
and
$x_s=0$ for all $s\neq i,j$ must be interchanged, and the triangle
$(x),(x'),(x'')$ with 
\newline
\centerline{$x_i=x'_i=x''_i=9,\; \{x_j,x'_j,x''_j\}= -\calD\setminus \{0\}$}
and
$z_s=0$ for all $s\neq i,j$, is preserved by $F$ (c.f. \ref{singer:size:4} and see 
\ref{fig:4sizelines}).
\newline
From \ref{singer:jointriangle} $F$ particularly interchanges points:
\begin{itemize}\def\labelitemi{\strut}\itemsep-2pt
  \item
   $(a)\in S_i$ with $(b)\in S_i$, and $a_i=1, b_i=3$;
  \item
   $(a')\in S_i$ with $(b')\in S_i$, and $a'_i=2, b'_i=11$
   \item
    $(a'')\in S_i$ with $(b'')\in S_i$, and $a''_i=5, b''_i=7$ .
\end{itemize}
The points $(a'')$, $(b'')$ lay on the line $[y']$ (which remains invariant), 
and thus $F$ fixes $(y')$ -- the third point (besides $(\theta)$) from
$S_i$ laying on $[y']$. From the above note that the set $\{(x)\in S_i\colon x_i\in\{1,2,4\}\}=:L$ consits
of three collinear points, but  $F(L)=\{(x)\in S_i\colon x_i\in\{3,11,4\}\}$ does not.
\par
Consequently, the line $[z]$ with 
$z_i=z_j=1$ must be mapped onto itself.
Then, again using \ref{singer:size:4} and \ref{singer:jointriangle}, step by step, we come to 
$F\restriction{\otocz({\goth N}^n,{(\theta)})}=\id$.
\end{auproof} 
In the similar way, as in the case of Fano plane, we can prove the following facts:
\begin{auxlem}\label{singer:au:1}
  Let $F$ be an automorphism of $\goth M$ such that $F$ leaves every line 
  through a point $a$ invariant.
  Then 
  $F\restriction{\otocz({\goth N}^n,{a})}=\id$.
\end{auxlem}
\begin{auproof}
  Without loss of generality we can assume that $a = \theta$
  (consider $F':= \tau_{-a} \circ F \circ \tau_a$, if necessary).
  In particular $F$ leaves every line in $\flines^{\alpha}$ invariant.
  Then we apply \ref{singer:size:4} and get our claim.
\end{auproof}
\begin{auxlem}\label{singer:au:2}
  Let $a,b$ be two points of ${\goth N}^n$ such that 
  $b$ is a point of a cyclic projective subplane
  in ${\otocz({\goth N}^n,{a})}$.
  If $F$ is an automorphism of ${\goth N}^n$ such that
  $F\restriction{\otocz({\goth N}^n,{a})}=\id$
  then
  $F\restriction{\otocz({\goth N}^n,{b})}=\id$ as well.
\end{auxlem}
\begin{auproof}
Assume that $a = (\theta)$.
The degree of every point $(x)\in S_i$  in $\otocz({\goth N}^n,{(\theta)})$ amount to $3n+1$ --
follows from \ref{singer:jointriangle}. Thus, every line through $b$ is
preserved and the claim follows from \ref{singer:au:1}.
\end{auproof}
       Directly from \ref{pr:autinsegre} and \ref{lem:syminsegre} we get:
     \begin{auxlem}\label{lem:syminpowerplane}
       Let $\sigma\in S_n$. We define the map $h_{\sigma}$ on $(C_{13})^n$ 
       by the formula
       $h_{\sigma}((x_1,\ldots,x_n)) = (x_{\sigma(1)},\ldots,x_{\sigma(n)})$.
       Then $G_{\sigma}=(h_{\sigma},h_{\sigma})\in \Aut({\goth N}^n)$, $G_{\sigma}(\theta) = \theta$, 
       and $G_{\sigma}(P_i)= P_{\sigma(i)}$ for all $i=1,\ldots,n$. 
 \end{auxlem}
 Let us consider $G=G_{\sigma}^{-1}\circ F$; then from \ref{lem:syminpowerplane} 
 $G \in \otocz({\Aut({\goth N}^n)},{(\theta)})$
 and $G(P_i)=P_i$.
 \par
  Observe the group $H^n :=\{(h_1,\ldots,h_n) \colon h_1,\ldots,h_n \in H \}$
  of permutations of $(C_{13})^n$ with the set $H$ defined in \eqref{wz:graut:C13}.
  The  coordinatewise action is evident, namely: 
  \newline
  \centerline{
  $(h_1,\ldots,h_n)(x_1,\ldots,x_n) = (h_1(x_1),\ldots,h_n(x_n))$.}
  In view of \ref{pr:autinsegre}, $H^n$ is an automorphism group of ${\goth N}^n$
  that leaves every "projective part" $P_i$ through $(\theta)$ invariant, and permutes lines through
$(\theta)$. 
\par
Take  $h\in H^n$.
We use the map $G'=h^{-1}\circ G$ to make every line in $\flines^4$ invariant
in the case of $G$ does not preserve some line in $\flines^4$.
Hence, from \ref{singer:fourpreserve}, $G'\in \otocz({\Aut({\goth N}^n}),{(\theta)})$ is the identity on
$\otocz({\goth N}^n,{(\theta)})$.
In a view of \ref{singer:au:2}, \ref{connect:fanem}, and by induction we obtain
$G'=\id$.
Finally, to make the proof complete we note:
\begin{eqnarray*}
 G_\sigma \circ \tau_u \circ G_{\sigma}^{-1} = \tau_{G_\sigma(u)}, &
 G_\sigma \circ h\circ G_{\sigma}^{-1} = G_\sigma(h), &
h \circ \tau_u \circ h^{-1} = \tau_{h(u)}.
\end{eqnarray*}
\end{proof}

\end{document}